\documentclass[twocolumn]{autart}    

\newcommand{\D}{\mathcal{D}}

\newcommand{\E}{\mathcal{E}}

\newcommand{\V}{\mathcal{V}}
\newcommand{\Gr}{\mathcal{G}}
\newcommand{\W}{\mathcal{W}}
\newcommand{\Sig}{\Sigma}

\newcommand{\Ncal}{\mathcal{N}}
\newcommand{\Xcal}{\mathcal{X}}
\newcommand{\Ucal}{\mathcal{U}}
\newcommand{\Ycal}{\mathcal{Y}}

\newtheorem{theorem}{Theorem}[section]

\newtheorem{corollary}[theorem]{Corollary}
\newtheorem{definition}{Definition}
\newtheorem{proposition}[theorem]{Proposition}
\newtheorem{assumption}{Assumption}
\newtheorem{example}[theorem]{Example}

\usepackage{times,amsmath,amssymb,graphicx,tikz,algorithm,algorithmic,url,hhline,booktabs}
\usetikzlibrary{automata,arrows,positioning}

\usetikzlibrary{shadows}
\usepackage[colorlinks,citecolor=blue]{hyperref}

\newenvironment{proof}[1][Proof]{\textbf{#1.} }{\ \rule{0.5em}{0.5em}}

\usepackage{tcolorbox}
\usepackage{tabularx}
\usepackage{array}
\usepackage{colortbl}
\tcbuselibrary{skins}

\newcolumntype{Y}{>{\raggedleft\arraybackslash}X}

\tcbset{tab1/.style={fonttitle=\bfseries\large,fontupper=\normalsize\sffamily,
colback=yellow!10!white,colframe=red!75!black,colbacktitle=Salmon!40!white,
coltitle=black,center title,freelance,frame code={
\foreach \n in {north east,north west,south east,south west}
{\path [fill=red!75!black] (interior.\n) circle (3mm); };},}}

\tcbset{tab2/.style={enhanced,fonttitle=\scriptsize,fontupper=\scriptsize,
colback=black!10!white,colframe=black,colbacktitle=Salmon!40!white,
coltitle=black,center title}}

\begin{document}

\begin{frontmatter}

\title{Observability and reconstructibility of large-scale Boolean control networks via network aggregations\thanksref{footnoteinfo}} 

\thanks[footnoteinfo]{This work was supported by
National Natural Science Foundation of China (No. 61603109),
Natural Science Foundation of Heilongjiang Province of China (No. LC2016023),
and Fundamental Research Funds for the Central Universities (HEUCFM170406).
This paper was not presented at any conference.}

\author[HEU]{Kuize Zhang}\ead{zkz0017@163.com},    

\address[HEU]{College of Automation, Harbin Engineering University, Harbin, 150001, PR China}  


\begin{abstract}                          
	It is known that determining the observability and reconstructibility of Boolean control networks (BCNs) 
	are both {\bf NP}-hard
	in the number of nodes of BCNs.
	In this paper, we use the aggregation method to overcome the challenging complexity problem in
	verifying the observability and reconstructibility of large-scale BCNs 
	with special structures in some sense.
	First, we define a special class of aggregations that are compatible with observability and reconstructibility
	(i.e, observability and reconstructibility are meaningful for each part of the aggregation),
	and show that even for this special class of aggregations, the whole BCN being observable/reconstructible does not 
	imply the resulting sub-BCNs being observable/reconstructible, and vice versa.
	Second, for acyclic aggregations in this special 
	class, we prove that all resulting sub-BCNs being observable/reconstructible implies the whole BCN being observable/reconstructible.
	Third, we show that finding such acyclic special aggregations with sufficiently small parts can tremendously 
	reduce computational complexity. Finally, we use the BCN T-cell receptor kinetics model 
	to illustrate the efficiency of these results.
	In addition, the special aggregation method characterized in this paper can also be used to deal with 
	the observability/reconstructibility of large-scale linear (special classes of nonlinear) control systems
	with special network structures. 
\end{abstract}

\end{frontmatter}

\section{Introduction}
\label{sec:Introduction}

Boolean networks (BNs), initiated by Kauffman \cite{Kauffman1969RandomBN} to model genetic regulatory 
networks (GRNs) in 1969, are a class of discrete-time discrete-space dynamical systems.
In a BN, nodes
can be in one of two discrete values ``$1$'' and ``$0$'', which represent
the state ``on'' of genes/high concentration of proteins and the
state ``off'' of genes/low concentration of proteins,
respectively. Every node updates its values according to the values of the nodes of
the network. When external regulation or perturbation is considered, 
BNs are naturally extended to Boolean control networks (BCNs) \cite{Ideker2001}.
Although a BN/BCN is a simplified model of GRNs, 
it can be used to characterize many
important phenomena of biological systems, e.g., cell cycles \cite{Faure2006},
cell apoptosis \cite{Sridharan2012}, and so on. Hence the study on BNs/BCNs has
been paid wide attention \cite{Kitano2002,Akutsu2000,Albert2000,Zhang2016AppInvettoMammalianCellCycle,Zhao2013AggregationAlgorithmBNattractor}.

The study on control-theoretic problems of BCNs can date back to 2007 
\cite{Akutsu2007}, in which the problem of determining the controllability of BCNs is proved 
{\bf NP}-hard in the number of nodes, even restricted to some special classes of BCNs.
In addition, it is pointed out that ``One of the major goals of systems biology is
to develop a control theory for complex biological systems'' in \cite{Akutsu2007}.
Since then, especially since the control-theoretic problem framework for BCNs is constructed 
in \cite{Cheng2009bn_ControlObserva} based on a new tool that is called the semi-tensor product
(STP) of matrices in 2009, 
many basic control-theoretic problems of BCNs have been characterized, e.g.,
controllability
\cite{Cheng2009bn_ControlObserva,Zhao2010InputStateIncidenceMatrix},
observability \cite{Cheng2009bn_ControlObserva,Zhao2010InputStateIncidenceMatrix,Fornasini2013ObservabilityReconstructibilityofBCN,Li2015ControlObservaBCN,Zhang2014ObservabilityofBCN},
reconstructibility \cite{Fornasini2013ObservabilityReconstructibilityofBCN,Zhang2015WPGRepresentationReconBCN},
identifiability \cite{Cheng2011IdentificationBCN}, invertibility \cite{Zhang2015Invertibility_BCN},
Kalman decomposition \cite{Zou2015KalmanDecompositionBCN}, and
other related works \cite{Li2015LogicalMatrixFactorizationBN,Liu2015,Li2014ControllabilityBCNAvoidingState,Li2016PinningControlSynchronizationBCN},
just to name a few.
Similar to controllability, observability and reconstructibility are also basis control-theoretic problems,
where observability/reconstructibility means using input sequences and the corresponding output sequences to 
determine the initial/current states.
The problems of determining the observability and reconstructibility of BCNs are
also {\bf NP}-hard in the number of nodes
\cite{Laschov2013ObservabilityofBN:GraphApproach,Zhang2015WPGRepresentationReconBCN}.
Hence checking both controllability and observability/reconstructibility of BCNs are computationally intractable.
However, unlike controllability, BCNs possess different types of observability/reconstructibility.
In \cite{Zhang2014ObservabilityofBCN}, the relationship between all four types of observability of BCNs
in the literature is studied, and it is proved that no two of 
them are equivalent, revealing the nonlinearity of BCNs.
In \cite{Zhang2015WPGRepresentationReconBCN}, except for giving a new equivalent condition for 
the reconstructibility studied in \cite{Fornasini2013ObservabilityReconstructibilityofBCN},
a more general reconstructibility is characterized.
In the investigation of controllability/observability/reconstructibility,
the primary point is to design fast algorithms for
verifying them. In \cite{Cheng2009bn_ControlObserva,Zhao2010InputStateIncidenceMatrix}, the algorithms 
for verifying controllability run in exponential time in the number of nodes. 
In \cite{Zhang2014ObservabilityofBCN}, a unified method based on finite automata is proposed to verify
all four different types of observability in the literature. Although 
the algorithms designed in \cite{Zhang2014ObservabilityofBCN} perform not worse than
the ones designed in \cite{Fornasini2013ObservabilityReconstructibilityofBCN,Li2013ObservabilityConditionsofBCN,Li2015ControlObservaBCN} in the worse case,
they all run in exponential time in the number of nodes.
The algorithms in \cite{Fornasini2013ObservabilityReconstructibilityofBCN,Zhang2015WPGRepresentationReconBCN}
for verifying reconstructibility also run in exponential time in the number of nodes.
Hence, in order to find fast algorithms for verifying controllability/observability/reconstructibility,
an attemptable way is to consider 
BCNs with special network structures.
The method of aggregating (i.e., partitioning) network graphs that has been widely used in pagerank algorithms
\cite{Ishii2012AggregationPageRank}, social networks \cite{Louati2013AggregationSocialNetworks},
BNs/BCNs \cite{Zhao2013AggregationAlgorithmBNattractor,Zhao2016ControllabilityAggregationBCN},
etc., may provide an effective way.

It is {\bf NP}-complete to check whether 
a BN has a fixed point (cf. \cite{Zhao2005NPHardnessFixedPointBN}).
Hence similar to verifying controllability and observability/reconstructibility, it is also computationally 
intractable to check the existence of fixed points for a BN. In
\cite{Zhao2013AggregationAlgorithmBNattractor}, an efficient way to find attractors is proposed based on 
aggregating a BN; particularly, for an acyclic aggregation,
an efficient way to find all attractors is proposed
by composing the attractors of each part of a BN. Similar idea has been used to deal with the 
controllability and stabilizability of BCNs \cite{Zhao2016ControllabilityAggregationBCN}.
Since BCNs have external nodes (i.e., input nodes), the results related to controllability/stabilizability based on
aggregations are not so 
perfect as the counterpart for BNs in \cite{Zhao2013AggregationAlgorithmBNattractor}.
Actually, one cannot always verify the controllability/stabilizability of a BCN
by verifying the controllability/stabilizability of the parts of the BCN.
It is proved in \cite{Zhao2016ControllabilityAggregationBCN} that if a BCN
is controllable then all parts of the BCN are controllable for any aggregation of the BCN each of whose
parts has at least one state node. However, in this paper we will show that this conclusion does not hold for 
observability/reconstructibility even for a more special class of aggregations than the one used in 
\cite{Zhao2016ControllabilityAggregationBCN}.
Nevertheless, we will still study whether the aggregation method can be used to
deal with the observability/reconstructibility of BCNs,
since for acyclic aggregations, the method of stabilizing the whole BCN
by stabilizing some parts of the BCN given in \cite{Zhao2016ControllabilityAggregationBCN} may 
tremendously reduce computational complexity under certain sufficient conditions
that the number of nodes of each part
is much smaller than that of the nodes of the whole BCN.
Also due to the former essential differences between controllability and observability/reconstructibility,
we have to use different aggregations in this paper.
The advantage of the aggregation method has been shown by practical examples
in both \cite{Zhao2013AggregationAlgorithmBNattractor} and \cite{Zhao2016ControllabilityAggregationBCN},
where the BCN T-cell receptor kinetics model (cf. \cite{Klamt2006TCellReceptor})
is used to illustrate the effectiveness of the aggregation method.
The model has $37$ state nodes and $3$ input nodes,\footnote{In \cite{Zhao2013AggregationAlgorithmBNattractor},
in order to compute attractors, the $3$ input nodes are assumed to be constant.}
i.e., it has $2^{37}$ states, and $2^{3}$ inputs. It is almost impossible to use the general methods
given in \cite{Zhao2005NPHardnessFixedPointBN,Zhao2010InputStateIncidenceMatrix} to 
compute its attractors or check its controllability/stabilizability due to the speed limitation of electrical computers.
However, using the aggregation methods in \cite{Zhao2013AggregationAlgorithmBNattractor,Zhao2016ControllabilityAggregationBCN},
these two problems have been solved. In this paper, we will try to find special aggregations 
that are compatible with observability/reconstructibility, use the special aggregations to 
design fast algorithms for verifying the observability/reconstructibility
of large-scale BCNs with special structures, and also use the BCN T-cell receptor kinetics model
to illustrate the efficiency of our results.

The remainder of the paper is organized as below.
In Section \ref{sec:preliminary}, basic concepts on BCNs, observability, reconstructibility,
and new aggregations compatible with observability and reconstructibility are introduced.
In Section \ref{sec:mainresults_obser}, the observability results based on aggregations are
shown. First, we define a special class of aggregations that are compatible with observability 
(i.e, observability is meaningful for each part of the aggregation),
and show the relationship between the whole BCN being observable and all resulting sub-BCNs being observable.
There is no implication relation between them.
Second, for acyclic aggregations in this special class, we prove that all resulting sub-BCNs being
observable implies the whole BCN being observable.
Finally, approximate computational complexity based on the special aggregation method is analyzed,
showing that approximately, the more parts such an aggregation has, the more effective the method is.
In section \ref{sec:mainresults_recon}, similar results on reconstructibility is derived.
In section \ref{sec:application},
the BCN T-cell receptor kinetics model is used to illustrate the efficiency
of the main results given in Sections \ref{sec:mainresults_obser} and \ref{sec:mainresults_recon}.
Section \ref{sec:conclusion} is a short conclusion with further discussion.

\section{Preliminaries}\label{sec:preliminary}

\subsection{Boolean control networks}

A BCN is described by a directed graph that is called a network graph,
and logical equations (e.g., Eqn. \eqref{eqn1:observability_aggregation}),
where a network graph
consists of input nodes, state nodes, output nodes, and directed edges between nodes  (e.g., as shown in Fig. 
\ref{fig1:BCNnetworkgraph}).
In a network graph, each directed edge from node $v_i$ to node $v_j$ means that the value 
($1$ or $0$) of $v_j$ at time step $t+1$ is affected by the value of $v_i$ at time step $t$.
Note that from a network graph, 
one can only know whether or not a node is affected by another node.
In order to uniquely determine a BCN, 
logical equations are also needed to describe the specific updating rules.
Actually, logical equations uniquely determine a BCN. And furthermore, 
the BCNs represented by different logical equations may have the same network graph.
For example, the BCNs represented by the following equations \eqref{eqn1:observability_aggregation}
and \eqref{eqn2:observability_aggregation} both 
have the network graph as shown in Fig. \ref{fig1:BCNnetworkgraph}.

	\begin{equation}\label{eqn1:observability_aggregation}
		\left\{\begin{split}
			&A(t+1)=B(t)\wedge u(t),\\
			&B(t+1)=\neg A(t)\vee u(t),\\
			&y(t)=A(t),
		\end{split}\right.
	\end{equation}
	where $t=0,1,\dots$; $A(t),B(t),u(t),y(t)$ are Boolean variables ($1$ or $0$); $\wedge,\vee$, and $\neg$
	denote AND, OR, and NOT, respectively.

	\begin{equation}\label{eqn2:observability_aggregation}
		\left\{\begin{split}
			&A(t+1)=B(t)\bar\vee u(t),\\
			&B(t+1)=\neg A(t)\wedge u(t),\\
			&y(t)=A(t),
		\end{split}\right.
	\end{equation}
	where $t=0,1,\dots$; $A(t),B(t),u(t),y(t)$ are Boolean variables; $\bar\vee$ denotes XOR.

	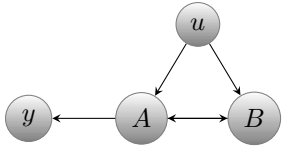
\begin{figure}
		\begin{center}
        \begin{tikzpicture}[->,>=stealth,node distance=1.5cm]
          \tikzstyle{node}=[shape=circle,draw=black!50,top color=white,bottom color=black!50]
          \tikzstyle{input}=[shape=circle,draw=black!50,top color=white,bottom color=black!50]
          \tikzstyle{output}=[shape=circle,draw=black!50,top color=white,bottom color=black!50]
          \tikzstyle{disturbance}=[shape=circle,fill=red!20!blue,draw=none,text=white,inner sep=2pt]
          \node[node] (A)                {$A$};
          \node[node] (B) [right of=A] {$B$};
          \node[input] (u) [above of=A,xshift=0.75cm,yshift=-.25cm] {$u$};
		  \node[output] (y) [left of=A] {$y$};
          \path   (A) edge (y)
                  (u) edge  (A)
				  (u) edge (B)
				  (A) edge (B)
				  (B) edge (A);
        \end{tikzpicture}
		\caption{Example of a network graph consisting of one input node $u$, two state nodes $A,B$,
		and one output node $y$.}
		\label{fig1:BCNnetworkgraph}
      \end{center}
	\end{figure}

In Fig. \ref{fig1:BCNnetworkgraph}, an input node has $0$ indegree (i.e., the number of entering edges at the node), 
an output node has $0$ outdegree (i.e., the number of leaving edges at the node),
and state nodes may have both positive indegree and positive outdegree.
Hereinafter, we denote $\D:=\left\{ 0,1 \right\}$; $[i,j]:=\left\{ i,i+1,\dots,j \right\}$
for integers $i\le j$;
$C_i^j:=\frac{i!}{j!(i-j)!}$ for positive integers $i\ge j$.
$2^S$ stands for the power set of a set $S$,
$\oplus$ and $\odot$ stand for the addition and multiplication modulo 2, respectively.
Generally, a BCN is formulated as in Eqn. \eqref{BCN:ObservabilityAggregation}:

\begin{equation}\label{BCN:ObservabilityAggregation}
\begin{split}
  &\left\{ \begin{split}
 &x_1 (t + 1) = f_1 (x_1 (t), \dots ,x_n (t),u_1 (t), \dots ,u_m (t)), \\
 &x_2 (t + 1) = f_2 (x_1 (t), \dots ,x_n (t),u_1 (t), \dots ,u_m (t)), \\
  &\vdots  \\
 &x_n (t + 1) = f_n (x_1 (t), \dots ,x_n (t),u_1 (t), \dots ,u_m (t)),
 \end{split} \right.\\
 &\left\{ \begin{split}
 &y_1 (t ) = h_1 (x_1 (t), \dots ,x_n (t)), \\
 &y_2 (t ) = h_2 (x_1 (t), \dots ,x_n (t)), \\
  &\vdots  \\
 &y_q (t ) = h_n (x_1 (t), \dots ,x_n (t)),
 \end{split} \right.\\
 \end{split}
\end{equation}
where $t=0,1,\dots$ denote time steps; $x_i(t),u_j(t)$, and $y_k(t)\in\D$ denote 
values of state node $x_i$, input node $u_j$, and output node $y_k$ at time step $t$,
respectively,
$i\in[1,n]$, $j\in[1,m]$, $k\in[1,q]$; $f_i:\D^{m+n}\to\D$, and $h_j:\D^{n}\to \D$
are Boolean functions, $i\in[1,n]$, $j\in[1,q]$.

Eqn. \eqref{BCN:ObservabilityAggregation} is represented in the compact form
\begin{equation}\label{BCN1:ObservabilityAggregation}
	\begin{split}
		&x(t+1) = f(x(t),u(t)),\\
		&y(t) = h(x(t)),
	\end{split}
\end{equation}
where $t=0,1,\dots$; $x(t)\in\D^{n}$, $u(t)\in\D^{m}$, and $y(t)\in\D^{q}$
stand for the state, input, and output of the BCN at time step $t$;
$f:\D^{n+m}\to \D^{n}$ and $h:\D^{n}\to \D^{q}$ are Boolean mappings.

\subsection{Observability of Boolean control networks}

In \cite{Zhang2014ObservabilityofBCN}, all four types of observability are 
characterized for BCNs. In this paper, we are particularly interested in the linear type
(also characterized in \cite{Fornasini2013ObservabilityReconstructibilityofBCN}).

\begin{definition}\label{def1:observability_aggregation}
	A BCN \eqref{BCN1:ObservabilityAggregation} is called observable if 
	for all different initial states $x_0,x_0'\in\D^{n}$, for each input sequence
	$\{u_0,u_1,\dots\}\subset\D^{m}$, the corresponding output sequences 
	$\{y_0,y_1,\dots\}$ and $\{y_0',y_1',\dots\}$ are different.
\end{definition}

We use a graph-theoretic method proposed in \cite{Zhang2014ObservabilityofBCN} to verify observability
in what follows.

\begin{definition}[\cite{Zhang2014ObservabilityofBCN}]\label{def2:observability_aggregation}
	Consider a BCN \eqref{BCN1:ObservabilityAggregation}.
	A weighted directed graph $\Gr_o=(\V,\E,\W,2^{\D^m})$ is
	called the observability weighted pair graph (OWPG) of the BCN if the vertex set $\V$ equals
	$ \{\{x,x'\}\in\D^n\times\D^n|h(x)=h(x')\}$, the edge set $\E$ equals
	$\{(\{x_1,x_1'\},\{x_2,x_2'\})\in\V\times\V|
	\text{there exists }u\in\D^m\text{ such that }f(x_1,u)=x_2\text{ and }f(x_1',u)=x_2',
	\text{ or, }f(x_1,u)=x_2'\text{ and }f(x_1',u)=x_2\}\subset \V\times\V$, and
	the weight function $\W:\E\to 2^{\D^m}$ maps each edge $(\{x_1,x_1'\},\{x_2,x_2'\})\in\E$ to $
	\{u\in\D^m|f(x_1,u)=x_2\text{ and }f(x_1',u)=x_2',
	\text{ or, }f(x_1,u)=x_2'\text{ and }f(x_1',u)=x_2	\}$. 
	A vertex $\{x,x'\}$ is called diagonal if $x=x'$, and called non-diagonal otherwise.
\end{definition}

\begin{proposition}[\cite{Zhang2014ObservabilityofBCN}]\label{prop1:observability_aggregation}
	A BCN \eqref{BCN1:ObservabilityAggregation} is not observable if and only if
	its OWPG has a non-diagonal vertex $v$ and a cycle $C$ such that 
	there is a path from $v$ to a vertex of $C$.
\end{proposition}

The computational cost of constructing the OWPG of a BCN \eqref{BCN1:ObservabilityAggregation}
is at most $(2^n+2^n(2^n-1)/2)2^m=2^{n+m}+2^{2n+m-1}-2^{n+m-1}$.
Hence the computational complexity of using Proposition \ref{prop1:observability_aggregation} to check 
observability is $O(2^{2n+m-1})$. On the other hand, the size of the network graph of a BCN is 
at most $n+m+q+mn+n(n+q)$, which is significantly smaller than the size of the OWPG of the BCN,
then is it possible to design an algorithm to check observability by using only the network graph? The answer is ``No'',
because there exist two BCNs that have the same network graph, one of which is observable, but the other of which 
is not observable (see BCNs \eqref{eqn1:observability_aggregation} and \eqref{eqn2:observability_aggregation}).
Note that for a BCN \eqref{BCN1:ObservabilityAggregation}, the subgraph $(\V_d,(\V_d\times\V_d)\cap \E)$
generated by the set $\V_d$ of all diagonal vertices of
its OWPG contains a cycle; and for each diagonal vertex $v\in\V_d$, there is a path from $v$ to some vertex of a cycle
in the subgraph.
Then the following corollary holds.
\begin{corollary}\label{cor1:observability_aggregation}
	Consider a BCN \eqref{BCN1:ObservabilityAggregation}. If in its OWPG 
	there is a path from a non-diagonal vertex to a diagonal vertex
	then the BCN is not observable.
\end{corollary}
For example, BCNs \eqref{eqn1:observability_aggregation} and \eqref{eqn2:observability_aggregation} have
the same network graph shown in Fig. \ref{fig1:BCNnetworkgraph}, \eqref{eqn1:observability_aggregation}
is not observable (see Fig. \ref{fig2:BCNnetworkgraph}) by Corollary \ref{cor1:observability_aggregation},
but \eqref{eqn2:observability_aggregation} is observable
(see Fig. \ref{fig3:BCNnetworkgraph}) by Proposition \ref{prop1:observability_aggregation}.

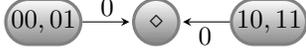
\begin{figure}
        \centering
\begin{tikzpicture}[shorten >=1pt,auto,node distance=1.5 cm, scale = 1.0, transform shape,
	->,>=stealth,inner sep=2pt,state/.style={
	rectangle,minimum size=6mm,rounded corners=3mm,
	very thick,draw=black!50,
	top color=white,bottom color=black!50,font=\ttfamily},
	point/.style={rectangle,inner sep=0pt,minimum size=2pt,fill=}]
	\node[state] (0001)                                 {$00,01$};
	\node[state] (*) [right of = 0001]               {$\diamond$};
	\node[state] (1011) [right of = *]               {$10,11$};

	\path [->] (0001) edge node {$0$} (*)
	      [->] (1011) edge node {$0$} (*)
		  ;
        \end{tikzpicture}
		\caption{Observability weighted pair graph of the BCN \eqref{eqn1:observability_aggregation},
		where $\diamond$ denotes the subgraph generated by all diagonal vertices.}
	\label{fig2:BCNnetworkgraph}
\end{figure}

\begin{figure}
        \centering
\begin{tikzpicture}[shorten >=1pt,auto,node distance=1.5 cm, scale = 1.0, transform shape,
	->,>=stealth,inner sep=2pt,state/.style={
	rectangle,minimum size=6mm,rounded corners=3mm,
	very thick,draw=black!50,
	top color=white,bottom color=black!50,font=\ttfamily},
	point/.style={rectangle,inner sep=0pt,minimum size=2pt,fill=}]
	\node[state] (0001)                                 {$00,01$};
	\node[state] (1011) [right of = 0001]               {$10,11$};
	\node[state] (*) [right of = 1011]               {$\diamond$};

        \end{tikzpicture}
		\caption{Observability weighted pair graph of the BCN \eqref{eqn2:observability_aggregation},
		where $\diamond$ denotes the subgraph generated by all diagonal vertices.}
	\label{fig3:BCNnetworkgraph}
\end{figure}
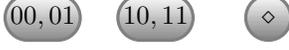

\subsection{Reconstructibility of Boolean control networks}

In this paper, we consider the reconstructibility of BCNs also of linear type
(cf. \cite{Zhang2015WPGRepresentationReconBCN,Fornasini2013ObservabilityReconstructibilityofBCN}).

\begin{definition}\label{def1:reconstructibility_aggregation}
	A BCN \eqref{BCN1:ObservabilityAggregation} is called reconstructible if there exists a positive integer $p$ 
	such that
	for all different initial states $x_0,x_0'\in\D^{n}$, for each input sequence
	$\{u_0,u_1,\dots,u_p\}\subset\D^{m}$, if the corresponding states $x_{p+1}$ and $x'_{p+1}$
	at time step $p+1$ are different, then the corresponding output sequences 
	$\{y_0,y_1,\dots,y_{p+1}\}$ and $\{y_0',y_1',\dots,y'_{p+1}\}$ are different.
\end{definition}

One directly sees that if a BCN \eqref{BCN1:ObservabilityAggregation} is observable then it is also 
reconstructible. However, the converse does not hold.
We also use a graph-theoretic method proposed in \cite{Zhang2015WPGRepresentationReconBCN} to verify reconstructibility.

\begin{definition}[\cite{Zhang2015WPGRepresentationReconBCN}]\label{def2:reconstructibility_aggregation}
	Consider a BCN \eqref{BCN1:ObservabilityAggregation}.
	A weighted directed graph $\Gr_r=(\V,\E,\W,2^{\D^m})$ is
	called the reconstructibility weighted pair graph (RWPG) of the BCN if the vertex set $\V$ equals
	$ \{\{x,x'\}\in\D^n\times\D^n|x\ne x',h(x)=h(x')\}$, the edge set $\E$ equals
	$\{(\{x_1,x_1'\},\{x_2,x_2'\})\in\V\times\V|
	\text{there exists }u\in\D^m\text{ such that }f(x_1,u)=x_2\text{ and }f(x_1',u)=x_2',
	\text{ or, }f(x_1,u)=x_2'\text{ and }f(x_1',u)=x_2\}\subset \V\times\V$, and
	the weight function $\W:\E\to 2^{\D^m}$ maps each edge $(\{x_1,x_1'\},\{x_2,x_2'\})\in\E$ to $
	\{u\in\D^m|f(x_1,u)=x_2\text{ and }f(x_1',u)=x_2',
	\text{ or, }f(x_1,u)=x_2'\text{ and }f(x_1',u)=x_2	\}$. 
\end{definition}

Note that for a BCN \eqref{BCN1:ObservabilityAggregation}, what differentiates its 
OWPG and RWPG is the vertex set.

\begin{proposition}[\cite{Zhang2015WPGRepresentationReconBCN}]\label{prop1:reconstructibility_aggregation}
	A BCN \eqref{BCN1:ObservabilityAggregation} is not reconstructible if and only if
	its RWPG has a cycle.
\end{proposition}

The computational cost of constructing the RWPG of a BCN \eqref{BCN1:ObservabilityAggregation}
is at most $(2^n(2^n-1)/2)2^m=2^{2n+m-1}-2^{n+m-1}$.
Hence the computational complexity of using Proposition \ref{prop1:reconstructibility_aggregation} to check 
reconstructibility is $O(2^{2n+m-1})$. On the other hand, the same as observability,
one cannot only use the network graph 
to check the reconstructibility of BCNs either, since there also exist two BCNs with the same network graph such 
that one BCN is reconstructible, but the other one is not reconstructible. Consider the following two
BCNs:
	\begin{equation}\label{eqn1:reconstructibility_aggregation}
		\left\{\begin{split}
			&A(t+1)=B(t)\wedge u(t),\\
			&B(t+1)=\neg A(t)\vee u(t),\\
			&y(t)=A(t)\bar\vee B(t),
		\end{split}\right.
	\end{equation}
	where $t=0,1,\dots$; $A(t),B(t),u(t),y(t)$ are Boolean variables.

	\begin{equation}\label{eqn2:reconstructibility_aggregation}
		\left\{\begin{split}
			&A(t+1)=B(t)\bar\vee u(t),\\
			&B(t+1)=A(t)\bar\vee u(t),\\
			&y(t)=A(t)\bar\vee B(t),
		\end{split}\right.
	\end{equation}
	where $t=0,1,\dots$; $A(t),B(t),u(t),y(t)$ are Boolean variables.

BCNs \eqref{eqn1:reconstructibility_aggregation} and \eqref{eqn2:reconstructibility_aggregation} have
the same network graph. By Proposition \ref{prop1:reconstructibility_aggregation}, \eqref{eqn1:reconstructibility_aggregation}
is reconstructible (see Fig. \ref{fig12:BCNnetworkgraph}),
but \eqref{eqn2:reconstructibility_aggregation} is not reconstructible
(see Fig. \ref{fig13:BCNnetworkgraph}).

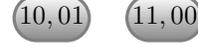
\begin{figure}
        \centering
\begin{tikzpicture}[shorten >=1pt,auto,node distance=1.5 cm, scale = 1.0, transform shape,
	->,>=stealth,inner sep=2pt,state/.style={
	rectangle,minimum size=6mm,rounded corners=3mm,
	very thick,draw=black!50,
	top color=white,bottom color=black!50,font=\ttfamily},
	point/.style={rectangle,inner sep=0pt,minimum size=2pt,fill=}]
	\node[state] (1001)                                 {$10,01$};
	\node[state] (1100) [right of = 1001]               {$11,00$};

        \end{tikzpicture}
		\caption{Reconstructibility weighted pair graph of the BCN \eqref{eqn1:reconstructibility_aggregation}.}
	\label{fig12:BCNnetworkgraph}
\end{figure}

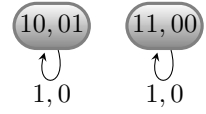
\begin{figure}
        \centering
\begin{tikzpicture}[shorten >=1pt,auto,node distance=1.5 cm, scale = 1.0, transform shape,
	->,>=stealth,inner sep=2pt,state/.style={
	rectangle,minimum size=6mm,rounded corners=3mm,
	very thick,draw=black!50,
	top color=white,bottom color=black!50,font=\ttfamily},
	point/.style={rectangle,inner sep=0pt,minimum size=2pt,fill=}]

	\node[state] (1001)                                 {$10,01$};
	\node[state] (1100) [right of = 1001]               {$11,00$};

	\path (1001) [loop below] edge node {$1,0$} (1001)
		  (1100) [loop below] edge node {$1,0$} (1100)
	;

        \end{tikzpicture}
		\caption{Reconstructibility weighted pair graph of the BCN \eqref{eqn2:reconstructibility_aggregation}.}
		\label{fig13:BCNnetworkgraph}
\end{figure}

\subsection{Aggregations of Boolean control networks}

For a BCN, let us denote the set of nodes of its network graph by
$\Ncal = \{x_1,\dots,x_n,u_1,\dots,u_m,y_1,\dots,y_q\}$, the set of state nodes
by $\Xcal =\{x_1, \dots, x_n\}$, the set of input nodes by
$\Ucal = \{u_1, \dots , u_m\}$, and the set of output nodes by $\Ycal=\{y_1,\dots,y_q\}$.
The nodes can be partitioned into $s$ blocks as follows:
\begin{equation}\label{partition:observability_aggregation}
	\Ncal=\Ncal_1\cup\cdots\cup\Ncal_s,
\end{equation}
where each $\Ncal_i$ is a nonempty proper subset of $\Ncal$, and $\Ncal_i\cap\Ncal_j=\emptyset$
for all $i\ne j$, $i,j\in[1,s]$. Each partition \eqref{partition:observability_aggregation} is called an aggregation of
the network graph. Note that in a BCN \eqref{BCN1:ObservabilityAggregation}, $\Ucal$ can be empty,
meaning that only a unique constant input sequence can be fed into the BCN.
Hereinafter we assume that neither $\Xcal$ nor
$\Ycal$ can be empty. If $\Ycal$ is empty, then one cannot observe any information of states of the BCN.
If $\Xcal$ is empty, it is meaningless to observe the BCN.
For an aggregation \eqref{partition:observability_aggregation}, each part $\Ncal_i$ is regarded as a super node,
then the aggregation is regarded as a new directed graph that is called an aggregation graph,
where the edge set consists of edges of the network graph whose tails and heads belong to different parts.
Also, each super node $\Ncal_i$ is regarded as a sub-BCN, denoted by $\Sig_i$.
For each super node $\Ncal_i$, 
its indegree (resp. outdegree) is the sum of edges entering (resp. leaving) $\Ncal_i$ in the aggregation graph, $i\in[1,s]$.
The purpose of aggregating the network graph of the BCN \eqref{BCN1:ObservabilityAggregation}
is to verify the observability/reconstructibility of the BCN by verifying the observability/reconstructibility
of its sub-BCNs.
So in order to reduce computational cost, the size of sub-BCNs should be as small as possible.

In \cite{Zhao2016ControllabilityAggregationBCN}, only controllability is considered, so the BCN 
\eqref{BCN1:ObservabilityAggregation} considered in \cite{Zhao2016ControllabilityAggregationBCN}
has an empty set $\Ycal$ of output nodes. Under the assumption that in an aggregation 
\eqref{partition:observability_aggregation}, each super node $\Ncal_i$ is weakly connected and contains at least
one state node, it is proved that the BCN is controllable only if each sub-BCN $\Sig_i$ is controllable, but
the converse is not true. One directly sees that without the assumption that each super node is 
weakly connected, all results in \cite{Zhao2016ControllabilityAggregationBCN} remain valid.
For observability/reconstructibility, since we must consider a nonempty set $\Ycal$ of output nodes, 
we aggregate the network graph in more special ways. Later on, for observability/reconstructibility,
we will show somehow converse results compared to the controllability results given in 
\cite{Zhao2016ControllabilityAggregationBCN}. However, controllability is not dual to observability 
for BCNs due to the essence of nonlinearity \cite{Zhang2014ObservabilityofBCN}.

In order to make each super node $\Ncal_i$ be a BCN such that it is meaningful to verify its observability/reconstructibility,
we only consider an aggregation \eqref{partition:observability_aggregation} satisfying the following Assumption 
\ref{assu1:observability_aggregation} in this paper. Assumption \ref{assu1:observability_aggregation}
is stronger than the previous assumption used in \cite{Zhao2016ControllabilityAggregationBCN}.
However, in order not to break the logical equations of the whole BCN, we have to make this stronger assumption.
Under Assumption \ref{assu1:observability_aggregation}, the sub-BCN $\Sig_i$ corresponding to each super node $\Ncal_i$
is of the form \eqref{BCN:ObservabilityAggregation}, $i\in[1,s]$.

\begin{assumption}\label{assu1:observability_aggregation}
	For each $i\in[1,s]$,
\begin{enumerate}
	\item\label{item1:observability_aggregation} 
		(making observing $\Sig_i$ meaningful) $\Ncal_i\cap\Ycal\ne\emptyset$;
		if $\Ncal_i\cap\Xcal\ne\emptyset$, then in $\Ncal_i$, for each state node $x\in\Ncal_i\cap \Xcal$,
		there is a path from $x$ to an output node $y\in\Ncal_i\cap \Ycal$ such that all state nodes in the whole 
		network graph $\cup_{i=1}^{s}\Ncal_i$ that are parents of $y$ belong to $\Ncal_i$.
	\item\label{item3:observability_aggregation}
		(making controlling $\Sig_i$ meaningful)
		If $\Ncal_i$ has a positive indegree, then outside $\Ncal_i$,
		all tails of all edges of the network graph entering  $\Ncal_i$
		are regarded as input nodes of $\Sig_i$. 
		(Note that all these tails are state nodes or input nodes of the network graph.)
\end{enumerate}

\end{assumption}

	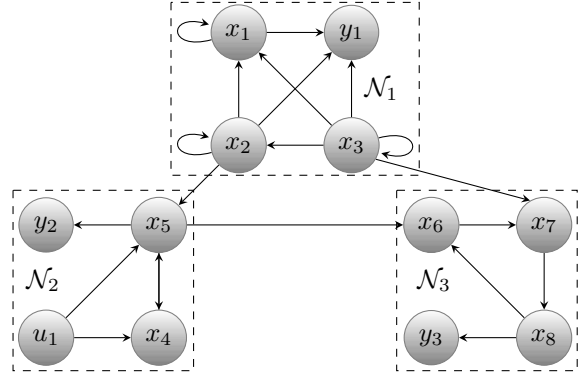
\begin{figure}
		\begin{center}
        \begin{tikzpicture}[->,>=stealth,node distance=1.5cm]
          \tikzstyle{node}=[shape=circle,draw=black!50,top color=white,bottom color=black!50]
          \tikzstyle{input}=[shape=circle,draw=black!50,top color=white,bottom color=black!50]
          \tikzstyle{output}=[shape=circle,draw=black!50,top color=white,bottom color=black!50]%
		  [shape=circle,fill=red!50!blue,draw=none,text=white,inner sep=2pt]
          \tikzstyle{disturbance}=[shape=circle,fill=red!20!blue,draw=none,text=white,inner sep=2pt]
          \node[node] (x1)                 {$x_1$};
          \node[output] (y1) [right of =x1]  {$y_1$};
		  \node[node] (x2) [below of =x1]  {$x_2$};
		  \node[node] (x3) [right of =x2]  {$x_3$};
		  \node[node] (x5) [below left of =x2]  {$x_5$};
		  \node[node] (x4) [below of =x5]  {$x_4$};
		  \node[input] (u1) [left of =x4]  {$u_1$};
		  \node[output] (y2) [above of =u1]  {$y_2$};
		  \node[node] (x6) [below right of =x3]  {$x_6$};
		  \node[node] (x7) [right of =x6]  {$x_7$};
		  \node[node] (x8) [below of =x7]  {$x_8$};
		  \node[output] (y3) [below of =x6]  {$y_3$};
          \path   (x1) edge (y1)
		  		  (x2) edge (x1)
				  (x3) edge (x1)
				  (x2) edge (y1)
				  (x3) edge (y1)
				  (x3) edge (x2)
		  		  (x1) [loop left] edge (x1)
				  (x2) [loop left] edge (x2)
				  (x3) [loop right] edge (x3);
				  \node at (1.9,-0.75) {$\Ncal_1$};
		 \path    (x2) edge (x5.45)
		          (u1) edge (x4)
				  (u1) edge (x5)
				  (x5) edge (y2)
				  (x4) edge (x5)
				  (x5) edge (x4);
				  \node at (-2.6,-3.3) {$\Ncal_2$};
		 \path    (x5.0) edge (x6.180)
		          (x3.-30) edge (x7.120)
				  (x6) edge (x7)
				  (x7) edge (x8)
				  (x8) edge (x6)
				  (x8) edge (y3);
				  \node at (2.6,-3.3) {$\Ncal_3$};
		 \draw [dashed] (-0.9,0.4) rectangle (2.4,-1.9)
		 			    (-3.0,-2.1) rectangle (-0.6,-4.5)
				        (2.1,-2.1) rectangle (4.5,-4.5)
		   ;
        \end{tikzpicture}
		\caption{Example of an aggregation of a BCN with $8$ state nodes, $1$ input node, and $3$ output nodes.}
		\label{fig4:BCNnetworkgraph}
      \end{center}
	\end{figure}

\begin{example}\label{exam1:observability_aggregation}
	Consider the following BCN corresponding to Fig. \ref{fig4:BCNnetworkgraph}.
	\begin{equation}\label{eqn3:observability_aggregation}
		\begin{split}
			&\Sig_1:\left\{  
			\begin{split}
				&x_1(t+1)=x_1(t)\oplus(x_2(t)\odot x_3(t)),\\
				&x_2(t+1)=x_2(t)\oplus x_3(t),\\
				&x_3(t+1)=x_3(t)\oplus 1,\\
				&y_1(t)=x_1(t)\odot(x_2(t)\oplus x_3(t)),
			\end{split}
			\right.\\
			&\Sig_2:\left\{  
			\begin{split}
				&x_4(t+1)=x_5(t)\odot u_1(t),\\
				&x_5(t+1)=x_4(t)\oplus u_1(t)\oplus x_2(t),\\
				&y_2(t)=x_5(t),
			\end{split}
			\right.
			\\
			&\Sig_3:\left\{
			\begin{split}
				&x_6(t+1)=x_8(t)\oplus x_5(t),\\
				&x_7(t+1)=x_6(t)\oplus x_3(t),\\
				&x_8(t+1)=x_7(t),\\
				&y_3(t)=x_8(t),
			\end{split}
			\right.
			\\
		\end{split}
	\end{equation}
	where $t=0,1,\dots$; $x_i(t),u_1(t),y_k(t)\in\D$, $i\in[1,8]$, $k\in[1,3]$.

	In Fig. \ref{fig4:BCNnetworkgraph}, all $\Ncal_1,\Ncal_2,\Ncal_3$ contain output nodes;
	$\Ncal_1$ and $\Ncal_3$ contain no input node; 
	$\Ncal_1$ contains edges $x_1\rightarrow y_1$, $x_2\rightarrow y_1$, and $x_3\rightarrow y_1$; 
	$\Ncal_2$ contains path $x_4\rightarrow x_5\rightarrow y_2$;
	$\Ncal_3$ contains path $x_6\rightarrow x_7\rightarrow  x_8\rightarrow y_3$;
	$x_2$ is an input node of $\Sig_2$; $x_3$ and $x_5$ are input nodes of $\Sig_3$;
	sub-BCNs $\Sig_1,\Sig_2$, and $\Sig_3$ in Eqn. \eqref{eqn3:observability_aggregation} correspond to
	the super nodes $\Ncal_1,\Ncal_2$, and $\Ncal_3$, respectively. Hence this aggregation satisfies Assumption
	\ref{assu1:observability_aggregation}. The corresponding aggregation graph is shown in 
	Fig. \ref{fig8:BCNnetworkgraph}.

	\begin{figure}
		\begin{center}
        \begin{tikzpicture}[->,>=stealth,node distance=1.5cm]
          \tikzstyle{node}=[shape=circle,draw=black!50,top color=white,bottom color=black!50]
          \tikzstyle{input}=[shape=circle,fill=red!80,draw=none,text=white,inner sep=2pt]
          \tikzstyle{output}=[shape=circle,fill=red!50!blue,draw=none,text=white,inner sep=2pt]
          \tikzstyle{disturbance}=[shape=circle,fill=red!20!blue,draw=none,text=white,inner sep=2pt]
          \node[node] (n1)                 {$\Ncal_1$};
          \node[node] (n2) [below left of =n1]  {$\Ncal_2$};
		  \node[node] (n3) [below right of =n1]  {$\Ncal_3$};
		  \path (n1) edge (n2)
		  		(n1) edge (n3)
				(n2) edge (n3)
				;
		  \end{tikzpicture}
		  \caption{Aggregation graph corresponding to Fig. \ref{fig4:BCNnetworkgraph}.}
		\label{fig8:BCNnetworkgraph}
      \end{center}
	\end{figure}
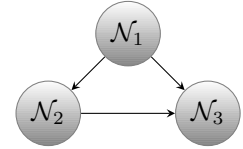
\end{example}

\section{Observability}\label{sec:mainresults_obser}

\subsection{Observability based on network aggregations}

In this subsection, we show whether one can verify the observability of a BCN \eqref{BCN1:ObservabilityAggregation}
via verifying the observability of its sub-BCNs obtained by aggregating its network graph under Assumption
\ref{assu1:observability_aggregation}.
First we investigate whether 
a BCN \eqref{BCN1:ObservabilityAggregation} being observable implies its sub-BCNs 
obtained by aggregating its network graph  under Assumption \ref{assu1:observability_aggregation} also being observable.
The following Example \ref{exam2:observability_aggregation} gives a negative answer.

\begin{example}\label{exam2:observability_aggregation}
	Consider the following BCN corresponding to Fig. \ref{fig5:BCNnetworkgraph}.
		\begin{equation}\label{eqn4:observability_aggregation}
		\begin{split}
			&\Sig_1:\left\{  
			\begin{split}
				&x_1(t+1)=u_1(t)\oplus x_2(t),\\
				&y_1(t)=x_1(t),
			\end{split}
			\right.\\
			&\Sig_2:\left\{  
			\begin{split}
				&x_2(t+1)=u_2(t)\oplus x_1(t),\\
				&x_3(t+1)=u_2(t)\oplus x_4(t),\\
				&y_2(t)=x_2(t)\odot x_3(t),
			\end{split}
			\right.
			\\
			&\Sig_3:\left\{  
			\begin{split}
				&x_4(t+1)=u_3(t)\oplus x_3(t),\\
				&y_3(t)=x_4(t),
			\end{split}
			\right.
			\\
		\end{split}
	\end{equation}
	where $t=0,1,\dots$; $x_i(t),u_j(t),y_k(t)\in\D$,
	$i\in[1,4]$, $j,k\in[1,3]$.
	
	It is not difficult to see that the aggregation in Fig. \ref{fig5:BCNnetworkgraph} satisfies Assumption
	\ref{assu1:observability_aggregation}.
	And sub-BCNs $\Sig_1,\Sig_2,\Sig_3$ in Eqn. \eqref{eqn4:observability_aggregation} correspond to
	the super nodes $\Ncal_1,\Ncal_2,\Ncal_3$ in Fig. \ref{fig5:BCNnetworkgraph}, respectively.
	$\Sig_1$ is observable, because $x_1(0)=y_1(0)$, and $y_1(0)$ can be observed. 
	Symmetrically $\Sig_3$ is also observable. In the OWPG of $\Sig_2$, we have an edge
	$\{00,01\}\xrightarrow[]{000}\{00,00\}$ from a non-diagonal vertex $\{00,01\}$ to a diagonal vertex
	$\{00,00\}$. Then by Corollary \ref{cor1:observability_aggregation}, $\Sig_2$ is not observable.
	Now consider the whole BCN \eqref{eqn4:observability_aggregation}. We have $x_1(0)=y_1(0)$,
	$x_2(0)=x_1(1)\oplus u_1(0)=y_1(1)\oplus u_1(0)$, $x_3(0)=x_4(1)\oplus u_3(0)=y_3(1)\oplus u_3(0)$, 
	$x_4(0)=y_3(0)$, $y_1(0),y_1(1),y_3(0),y_3(1)$ can be observed, $u_1(0)$ and $u_3(0)$ can be designed,
	hence \eqref{eqn4:observability_aggregation} is observable.

		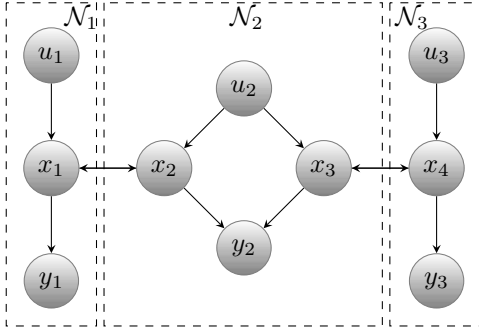
\begin{figure}
		\begin{center}
        \begin{tikzpicture}[->,>=stealth,node distance=1.5cm]
          \tikzstyle{node}=[shape=circle,draw=black!50,top color=white,bottom color=black!50]
          \tikzstyle{input}=[shape=circle,draw=black!50,top color=white,bottom color=black!50]
          \tikzstyle{output}=[shape=circle,draw=black!50,top color=white,bottom color=black!50]
          \tikzstyle{disturbance}=[shape=circle,fill=red!20!blue,draw=none,text=white,inner sep=2pt]
          \node[node] (x1)                 {$x_1$};
          \node[output] (y1) [below of =x1]  {$y_1$};
		  \node[input] (u1) [above of =x1]  {$u_1$};
		  \node[node] (x2) [right of =x1]  {$x_2$};
		  \node[input] (u2) [above right of = x2] {$u_2$};
		  \node[node] (x3) [below right of = u2] {$x_3$};
		  \node[output] (y2) [below left of =x3] {$y_2$};
		  \node[node] (x4) [right of =x3] {$x_4$};
		  \node[input] (u3) [above of =x4]  {$u_3$};
		  \node[output] (y3) [below of =x4]  {$y_3$};

		  \path (u1) edge (x1)
		  		(x1) edge (y1)
				(u3) edge (x4)
				(x4) edge (y3)
				(x2) edge (x1)
				(x1) edge (x2)
				(x3) edge (x4)
				(x4) edge (x3)
				(u2) edge (x2)
				(u2) edge (x3)
				(x2) edge (y2)
				(x3) edge (y2)
		  ;
		  \draw [dashed] (-0.6,2.2) rectangle (0.6,-2.1)
		  				 (0.7,2.2) rectangle (4.4,-2.1)
		  				 (4.5,2.2) rectangle (5.7,-2.1)
		  ;
		  \node at (0.4,2.0) {$\Ncal_1$};
		  \node at (2.6,2.0) {$\Ncal_2$};
		  \node at (4.8,2.0) {$\Ncal_3$};
		  \end{tikzpicture}
		\caption{Example of an aggregation of a BCN with $4$ state nodes, $3$ input nodes, and $3$ output nodes.}
		\label{fig5:BCNnetworkgraph}
      \end{center}
	\end{figure}

\end{example}

The aggregation graph corresponding to Fig. \ref{fig5:BCNnetworkgraph} contains cycles $\Ncal_1\leftrightarrow\Ncal_2$
and $\Ncal_2\leftrightarrow\Ncal_3$.
Then if an aggregation of a BCN \eqref{BCN1:ObservabilityAggregation} contains no cycle and
satisfies Assumption \ref{assu1:observability_aggregation}, is it true that 
a BCN \eqref{BCN1:ObservabilityAggregation} being observable implies its sub-BCNs 
also being observable?
The following Example \ref{exam4:observability_aggregation} gives a negative answer again.
\begin{example}\label{exam4:observability_aggregation}
	Consider the following BCN corresponding to Fig. \ref{fig7:BCNnetworkgraph}.
		\begin{equation}\label{eqn6:observability_aggregation}
		\begin{split}
			&\Sig_1:\left\{  
			\begin{split}
				&x_1(t+1)=x_2(t)\odot u_1(t),\\
				&x_2(t+1)=x_1(t),\\
				&y_1(t)=x_1(t),
			\end{split}
			\right.\\
			&\Sig_2:\left\{  
			\begin{split}
				&x_3(t+1)=x_2(t),\\
				&x_4(t+1)=x_3(t),\\
				&y_2(t)=x_4(t),
			\end{split}
			\right.
			\\
		\end{split}
	\end{equation}
	where $t=0,1,\dots$; $x_i(t),u_1(t),y_k(t)\in\D$, $i\in[1,4]$, $k\in[1,2]$.
	
	The aggregation shown in Fig. \ref{fig7:BCNnetworkgraph} satisfies Assumption
	\ref{assu1:observability_aggregation}, and the corresponding aggregation
	graph $\Ncal_1\rightarrow\Ncal_2$ contains no cycle.
	And sub-BCNs $\Sig_1,\Sig_2$ in Eqn. \eqref{eqn6:observability_aggregation} correspond to
	the super nodes $\Ncal_1,\Ncal_2$ in Fig. \ref{fig7:BCNnetworkgraph}, respectively.
	In the OWPG of $\Sig_1$, we have an edge $\{10,11\}\xrightarrow[]{0}\{01,01\}$ 
	from a non-diagonal vertex $\{10,11\}$ to a diagonal vertex $\{01,01\}$, then by Corollary
	\ref{cor1:observability_aggregation}, $\Sig_1$ is not observable.
	$\Sig_2$ is observable because $x_4(0)=y_2(0)$, $x_3(0)=x_4(1)=y_2(1)$,
	$y_2(0)$ and $y_2(1)$ can be observed. The whole BCN \eqref{eqn6:observability_aggregation} is
	observable because $x_1(0)=y_1(0)$, $x_2(0)=x_3(1)=x_4(2)=y_2(2)$,
	$x_3(0)=x_4(1)=y_2(1)$, $x_4(0)=y_2(0)$, 
	$y_1(0),y_2(0),y_2(1),$ and $y_2(2)$ can be observed.

		\begin{figure}
		\begin{center}
        \begin{tikzpicture}[->,>=stealth,node distance=1.5cm]
          \tikzstyle{node}=[shape=circle,draw=black!50,top color=white,bottom color=black!50]
          \tikzstyle{input}=[shape=circle,draw=black!50,top color=white,bottom color=black!50]
          \tikzstyle{output}=[shape=circle,draw=black!50,top color=white,bottom color=black!50]
          \tikzstyle{disturbance}=[shape=circle,fill=red!20!blue,draw=none,text=white,inner sep=2pt]
          \node[node] (x2)                 {$x_2$};
		  \node[node] (x1) [right of =x1]  {$x_1$};
          \node[output] (y1) [right of =x1]  {$y_1$};
		  \node[node] (x3) [below of =x2]  {$x_3$};
		  \node[node] (x4) [below of =x1]  {$x_4$};
		  \node[output] (y2) [right of =x4] {$y_2$};
		  \node[input] (u1) [above of =x1] {$u_1$};
		  \path 
				(x1) edge (x2)
				(x2) edge (x1)
		  		(x2) edge (x3)
				(x3) edge (x4)
				(x1) edge (y1)
				(x4) edge (y2)
				(u1) edge (x1)
		  ;
		  \draw [dashed] (-0.5,2.1) rectangle (3.5,-0.7)
		  				 (-0.5,-0.9) rectangle (3.5,-2.1)
		  ;
		  \node at (2.5,-0.5) {$\Ncal_1$};
		  \node at (2.5,-1.1) {$\Ncal_2$};
		  \end{tikzpicture}
		\caption{Example of an aggregation of a BCN with $4$ state nodes, $1$ input node, and $2$ output nodes.}
		\label{fig7:BCNnetworkgraph}
      \end{center}
	\end{figure}
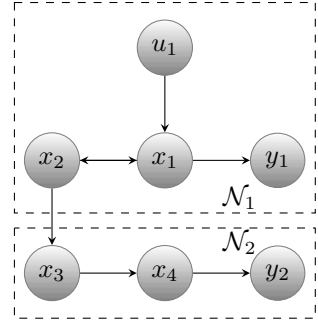

\end{example}

Next we discuss the opposite direction. That is, if an aggregation satisfies Assumption 
\ref{assu1:observability_aggregation}, whether all resulting sub-BCNs
being observable implies the whole BCN also being observable. Unfortunately, the answer is still negative.
The following Example \ref{exam3:observability_aggregation}
shows such an aggregation satisfying Assumption \ref{assu1:observability_aggregation},
containing a cycle and  satisfying that, even if all resulting sub-BCNs are
observable the whole BCN is not observable.

\begin{example}\label{exam3:observability_aggregation}
	Consider the following BCN corresponding to Fig. \ref{fig6:BCNnetworkgraph}.
	\begin{equation}\label{eqn5:observability_aggregation}
		\begin{split}
			&\Sig_1:\left\{  
			\begin{split}
				&x_1(t+1)=u_1(t),\\
				&x_2(t+1)=x_1(t)\oplus x_4(t),\\
				&y_1(t)=x_2(t),
			\end{split}
			\right.\\
			&\Sig_2:\left\{  
			\begin{split}
				&x_3(t+1)=x_1(t)\oplus x_4(t)\oplus 1,\\
				&x_4(t+1)=u_2(t),\\
				&y_2(t)=x_3(t),
			\end{split}
			\right.
			\\
		\end{split}
	\end{equation}
	where $t=0,1,\dots$; $x_i(t),u_j(t),y_k(t)\in\D$, $i\in[1,4]$, $j,k\in[1,2]$.
	
	The aggregation shown in Fig. \ref{fig6:BCNnetworkgraph} satisfies Assumption \ref{assu1:observability_aggregation},
	and its aggregation graph is a cycle $\Ncal_1\leftrightarrow\Ncal_2$.
	For $\Sig_1$, $x_2(0)=y_1(0)$, $x_1(0)=x_2(1)\oplus x_4(0)=y_1(1)\oplus x_4(0)$. 
	Since $y_1(0)$ and $y_1(1)$ can be observed and $x_4(0)$ is designable, $\Sig_1$ is observable.
	Similarly $\Sig_2$ is also observable. 
	Consider a non-diagonal vertex $\{0110,1111\}$ of the OWPG of \eqref{eqn5:observability_aggregation},
	we have an edge $\{0110,1111\}\xrightarrow[]{u_1u_2}\{u_101u_2,u_101u_2\}$, where 
	$u_1,u_2\in\D$, and $\{u_101u_2,u_101u_2\}$ is a diagonal vertex of the OWPG. Hence by Corollary 
	\ref{cor1:observability_aggregation}, \eqref{eqn5:observability_aggregation} is not observable.

		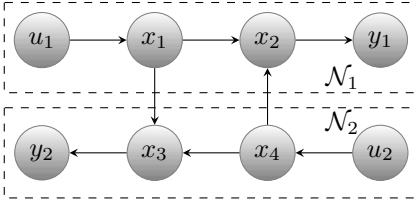
\begin{figure}
		\begin{center}
        \begin{tikzpicture}[->,>=stealth,node distance=1.5cm]
          \tikzstyle{node}=[shape=circle,draw=black!50,top color=white,bottom color=black!50]
          \tikzstyle{input}=[shape=circle,draw=black!50,top color=white,bottom color=black!50]
          \tikzstyle{output}=[shape=circle,draw=black!50,top color=white,bottom color=black!50]
          \tikzstyle{disturbance}=[shape=circle,fill=red!20!blue,draw=none,text=white,inner sep=2pt]
          \node[node] (x1)                 {$x_1$};
          \node[node] (x2) [right of =x1]  {$x_2$};
		  \node[node] (x3) [below of =x1]  {$x_3$};
		  \node[node] (x4) [right of =x3]  {$x_4$};
		  \node[output] (y2) [left of =x3] {$y_2$};
		  \node[input] (u1) [left of =x1] {$u_1$};
		  \node[output] (y1) [right of =x2] {$y_1$};
		  \node[input] (u2) [right of =x4] {$u_2$};
		  \path (x1) edge (x2)
		  		(x1) edge (x3)
				(x4) edge (x2)
		  		(x4) edge (x3)
				(u1) edge (x1)
				(x2) edge (y1)
				(x3) edge (y2) 
				(u2) edge (x4)
				;
		  \draw [dashed] (-2.0,0.5) rectangle (3.5,-0.7)
		  				 (-2.0,-0.9) rectangle (3.5,-2.1)
		  ;
		  \node at (2.5,-0.5) {$\Ncal_1$};
		  \node at (2.5,-1.1) {$\Ncal_2$};
		  \end{tikzpicture}
		\caption{Example of an aggregation of a BCN with $4$ state nodes, $2$ input nodes, and $2$ output nodes.}
		\label{fig6:BCNnetworkgraph}
      \end{center}
	\end{figure}
\end{example}

These two types of negative results show that observability possesses more complex properties than controllability,
as it is proved in \cite{Zhao2016ControllabilityAggregationBCN} that if the whole BCN is controllable, 
then all resulting sub-BCNs are controllable under a weaker assumption than Assumption 
\ref{assu1:observability_aggregation} stating that
each super node contains at least one state node.

These negative results seem to presage that one cannot use the aggregation method to verify observability.
However, the situation finally becomes positive.
Next we show when an aggregation satisfies 
Assumption \ref{assu1:observability_aggregation} and contains no cycle, all resulting sub-BCNs
being observable implies the whole BCN also being observable!
These results are put into the following subsection.

\subsection{Observability based on acyclic aggregations}

An aggregation \eqref{partition:observability_aggregation} is called acyclic if its aggregation graph contains no cycle.
For example, the aggregation shown in Fig. \ref{fig4:BCNnetworkgraph} (its aggregation graph is depicted in Fig.
\ref{fig8:BCNnetworkgraph}) is acyclic.

\begin{theorem}\label{thm1:observability_aggregation}
	Consider a BCN \eqref{BCN1:ObservabilityAggregation} that has
	an acyclic aggregation \eqref{partition:observability_aggregation} satisfying Assumption 
	\ref{assu1:observability_aggregation}.
	If all resulting sub-BCNs 
	are observable then
	the whole BCN is also observable.
\end{theorem}

\begin{proof}
	First we show that for an acyclic aggregation \eqref{partition:observability_aggregation},
	there is a reordering (i.e., a bijective) $\tau:[1,s]\to[1,s]$ such that for each $i\in[1,s]$,
	\begin{equation}\label{acyclic_partition:observability_aggregation}
		\Ncal^i:=\bigcup_{j=1}^{i}\Ncal_{\tau(j)}
	\end{equation} has zero indegree.

	Since the aggregation is acyclic, each subgraph of the aggregation graph $\Gr$ has a super node with indegree $0$.
	Suppose on the contrary that a subgraph $G$ of $\Gr$ has all super nodes with positive indegrees.
	Construct a new graph $G'$ by reversing the directions of all edges of $G$. Then $G'$ has a cycle,
	since $G'$ has finitely many nodes, and each node has a positive outdegree.
	Then $G$ and hence $\Gr$ have a cycle, which is a contradiction.

	Choose $k_1\in[1,s]$ such that $\Ncal_{k_1}$ has zero indegree in $\Gr$, remove $\Ncal_{k_1}$
	and all edges leaving $\Ncal_{k_1}$ from $\Gr$, and set $\tau(1):=k_1$.
	Then in the new $\Gr$, there is $k_2\in[1,s]\setminus\{k_1\}$ such that $\Ncal_{k_2}$ has zero indegree.
	Remove $\Ncal_{k_2}$ 
	and all edges leaving $\Ncal_{k_2}$ from the new $\Gr$, and set $\tau(2):=k_2$.
	Repeat this procedure until $\Gr$ becomes empty, we obtain a bijective $\tau:[1,s]\to[1,s]$ such that
	\eqref{acyclic_partition:observability_aggregation} holds.

	Second we show that if each sub-BCN $\Sig_i$ corresponding to $\Ncal_i$
	is observable then the whole BCN is also observable.
	Suppose for a BCN \eqref{BCN1:ObservabilityAggregation} that each resulting sub-BCN $\Sig_i$ is 
	observable, $i\in[1,s]$. Then for each given input sequence $\{u_0,u_1,\dots\}\subset\D^{m}$, 
	for all given different initial states $x_0,x_0'\in\D^n$, there is $k\in[1,s]$ such that 
	the components of $x_0,x_0'$ in $\Ncal_{\tau(k)}$ are not equal, and 
	the components of $x_0,x_0'$ in $\Ncal_{\tau(i)}$ are equal, $i\in[1,k-1]$.
	Note that
	\begin{equation}\label{eqn7:observability_aggregation}
		\begin{split}
			&\text{in the network graph, for all }i,j\in[1,s],\\&\text{if there exist node }
			v\in \Ncal_{\tau(i)}\text{ and node }v'\in\Ncal_{\tau(j)}\\
			&\text{such that }v\text{ affects }v'\text{ then }i\le j.
		\end{split}
	\end{equation}
	Then
	since sub-BCN $\Sig_{\tau(k)}$ is observable, the output sequences of $\Sig_{\tau(k)}$
	corresponding to the components of 
	$x_0,x_0'$ and input sequence $\{u_0,u_1,\dots\}$ in $\Ncal_{\tau(k)}$ are different.
	That is, the whole BCN is observable.
\end{proof}

In \cite{Zhao2016ControllabilityAggregationBCN}, an aggregation \eqref{partition:observability_aggregation}
satisfying \eqref{acyclic_partition:observability_aggregation} is called cascading;
and it is pointed out that each cascading aggregation is acyclic, which can also be seen by 
\eqref{eqn7:observability_aggregation}. Hence the following proposition follows from this property
and the proof of Theorem \ref{thm1:observability_aggregation}.
\begin{proposition}\label{prop2:observability_aggregation}
	An aggregation \eqref{partition:observability_aggregation} is acyclic if and only if it is cascading.
\end{proposition}

\begin{example}\label{exam5:observability_aggregation}
	Recall Example \ref{exam1:observability_aggregation}.
	The aggregation shown in Fig. \ref{fig4:BCNnetworkgraph} 
	is acyclic and satisfies Assumption \ref{assu1:observability_aggregation}.
	Next we show that the resulting sub-BCNs 
	$\Sig_1,\Sig_2,\Sig_3$ in \eqref{eqn3:observability_aggregation} are all observable.
	Then by Theorem \ref{thm1:observability_aggregation}, the whole BCN 
	\eqref{eqn3:observability_aggregation} is also observable.

	The OWPG of $\Sig_1$ has $8$ diagonal vertices, and $1+C_6^2=16$ non-diagonal vertices.
	The OWPG of $\Sig_1$ is shown in Fig. \ref{fig9:BCNnetworkgraph}.
	In Fig. \ref{fig9:BCNnetworkgraph}, there exists no path from a non-diagonal vertex to a diagonal vertex,
	and there exists no cycle in the subgraph generated by non-diagonal vertices.
	By Proposition \ref{prop1:observability_aggregation}, $\Sig_1$ is observable.

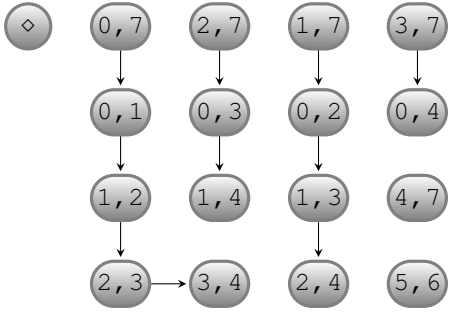
\begin{figure}
        \centering
\begin{tikzpicture}
	[shorten >=1pt,auto,node distance=1.5 cm, scale = 0.8, transform shape,
	->,>=stealth,inner sep=2pt,state/.style={
	rectangle,minimum size=6mm,rounded corners=3mm,
	very thick,draw=black!50,
	top color=white,bottom color=black!50,font=\ttfamily},
	point/.style={rectangle,inner sep=0pt,minimum size=2pt,fill=}]
	\matrix [nodes={fill=blue!20},column sep=0.5cm, row sep =0.5cm]
	{
	\node[state] {$\diamond$};
	&\node[state] (07) {0,7}; & \node[state] (27) {2,7}; & \node[state] (17) {1,7}; & \node[state] (37) {3,7}; \\
	&\node[state] (01) {0,1}; & \node[state] (03) {0,3}; & \node[state] (02) {0,2}; & \node[state] (04) {0,4}; \\
	&\node[state] (12) {1,2}; & \node[state] (14) {1,4}; & \node[state] (13) {1,3}; & \node[state] (47) {4,7}; \\
	&\node[state] (23) {2,3}; & \node[state] (34) {3,4}; & \node[state] (24) {2,4}; & \node[state] (46) {5,6}; \\
	};
	\path [->] (07) edge (01) 
	  		   (27) edge (03)
			   (17) edge (02)
   			   (37) edge (04)
			   (01) edge (12) 
	  		   (03) edge (14)
			   (02) edge (13)
   			   (13) edge (24)
			   (12) edge (23)
   			   (23) edge (34)
	;

        \end{tikzpicture}
		\caption{Observability weighted pair graph of the sub-BCN $\Sig_1$ in \eqref{eqn3:observability_aggregation},
		where $\diamond$ denotes the subgraph generated by all diagonal vertices, 
		numbers in circles are decimal representations for states of $\Sig_1$, formally,
		$0\sim 000$, $1\sim 001$, $2\sim 010$, $3\sim 011$, $4\sim 100$, $5\sim 101$, $6\sim 110$, $7\sim 111$.}
		\label{fig9:BCNnetworkgraph}
\end{figure}

	For $\Sig_2$, $x_5(0)=y_2(0)$, $x_4(0)=x_5(1)\oplus u_1(0)\oplus x_2(0)=y_2(1)\oplus u_1(0)\oplus x_2(0)$.
	$y_2(0)$ and $y_2(1)$ can be observed and $u_1(0)$ and $x_2(0)$ are designable, hence $\Sig_2$ is observable.

	For $\Sig_3$, $x_8(0)=y_3(0)$, $x_7(0)=x_8(1)=y_3(1)$, $x_6(0)=x_7(1)\oplus x_3(0)=x_8(2)\oplus x_3(0)=y_3(2)\oplus x_3(0)$.
	$y_3(0),y_3(1),$ and $y_3(2)$ can be observed and $x_3(0)$ is designable, hence $\Sig_3$ is also observable.

	The whole BCN \eqref{eqn3:observability_aggregation} has $2^8=256$ states, $2$ inputs, and $2^3=8$ outputs.
	Its OWPG has $(((1+C_6^2)*2+2^3)(2*2+2^2)((2C_4^2)*2+2^3)-2^8)/2=4992$ non-diagonal vertices,
	and $2^8=256$ diagonal vertices.
	It is much more complex to directly 
	use Proposition \ref{prop1:observability_aggregation} to check the observability
	of \eqref{eqn3:observability_aggregation} than using Theorem \ref{thm1:observability_aggregation}
	and Proposition \ref{prop1:observability_aggregation} to do it as above.
\end{example}


\subsection{Complexity analysis}

We analyze the computational complexity of using Theorem \ref{thm1:observability_aggregation}
and Proposition \ref{prop1:observability_aggregation} to determine the observability of the BCN
\eqref{BCN1:ObservabilityAggregation}. Following this way, we first find an acyclic aggregation 
of \eqref{BCN1:ObservabilityAggregation} that satisfies Assumption \ref{assu1:observability_aggregation},
then check the observability of all resulting sub-BCNs. If all resulting sub-BCNs are observable
then the whole BCN is observable.
Assume we have obtained an acyclic aggregation having $k$ parts with almost the same size and satisfying Assumption 
\ref{assu1:observability_aggregation}. Then each part approximately has $\frac{n}{k}$
state nodes and $\frac{m}{k}$ input nodes. The computational complexity is approximately
$k2^{\frac{2n+m}{k}-1}$ by Proposition \ref{prop1:observability_aggregation}.
For large-scale BCNs, $2n+m$ is huge. When $k<l(2n+m)$ for some positive
constant $l$, function $k2^{\frac{2n+m}{k}-1}$ is decreasing. Hence roughly speaking, 
the more parts an aggregation has and the more close the sizes of these parts are, the more effective the 
aggregation method is. It is hard to find such aggregations whose parts have approximately the same size,
but we can find aggregations having sufficiently many parts.
According to this rule, when aggregating a large-scale BCN, 
in order to reduce computational complexity as much as possible, one should 
make the parts as small as possible.

\section{Reconstructibility}\label{sec:mainresults_recon}

In this section, we study whether the aggregation method can be used to deal with the reconstructibility
of a large-scale BCN \eqref{BCN1:ObservabilityAggregation}. Partially due to the similarity between observability
and reconstructibility, similar results for reconstructibility are obtained.

First we give counterexamples to show that the whole BCN \eqref{BCN1:ObservabilityAggregation} being reconstructible
does not imply all resulting sub-BCNs being reconstructible, no matter whether an aggregation has a cycle;
and all resulting sub-BCNs being reconstructible does not imply the whole BCN being reconstructible if an aggregation has a cycle.

\begin{example}\label{exam1:reconstructibility_aggregation}
	Consider the following BCN corresponding to Fig. \ref{fig14:BCNnetworkgraph}.
		\begin{equation}\label{eqn3:reconstructibility_aggregation}
		\begin{split}
			&\Sig_1:\left\{  
			\begin{split}
				&x_1(t+1)=u_1(t)\oplus x_2(t),\\
				&y_1(t)=x_1(t),
			\end{split}
			\right.\\
			&\Sig_2:\left\{  
			\begin{split}
				&x_2(t+1)=x_3(t),\\
				&x_3(t+1)=x_2(t),\\
				&y_2(t)=x_2(t)\odot x_3(t),
			\end{split}
			\right.
			\\
			&\Sig_3:\left\{  
			\begin{split}
				&x_4(t+1)=u_2(t)\oplus x_3(t),\\
				&y_3(t)=x_4(t),
			\end{split}
			\right.
			\\
		\end{split}
	\end{equation}
	where $t=0,1,\dots$; $x_i(t),u_j(t),y_k(t)\in\D$,
	$i\in[1,4]$, $j\in[1,2]$, $k\in[1,3]$.
	
	The acyclic aggregation shown in Fig. \ref{fig14:BCNnetworkgraph} satisfies Assumption
	\ref{assu1:observability_aggregation}. In Example \ref{exam2:observability_aggregation}, we have 
	shown that $\Sig_1$ and $\Sig_3$ are both observable, hence they are both reconstructible.
	For $\Sig_2$, in its RWPG, there is a self-loop on non-diagonal vertex $\{10,01\}$, then by Proposition 
	\ref{prop1:reconstructibility_aggregation}, $\Sig_2$ is not reconstructible.
	Now consider the whole BCN \eqref{eqn3:reconstructibility_aggregation}. We have $x_1(0)=y_1(0)$,
	$x_2(0)=x_1(1)\oplus u_1(0)=y_1(1)\oplus u_1(0)$, $x_3(0)=x_4(1)\oplus u_2(0)=y_3(1)\oplus u_2(0)$, 
	$x_4(0)=y_3(0)$, $y_1(0),y_1(1),y_3(0),y_3(1)$ can be observed, $u_1(0)$ and $u_2(0)$ can be designed,
	hence \eqref{eqn3:reconstructibility_aggregation} is observable, and hence reconstructible.

		\begin{figure}
		\begin{center}
        \begin{tikzpicture}[->,>=stealth,node distance=1.5cm]
          \tikzstyle{node}=[shape=circle,draw=black!50,top color=white,bottom color=black!50]
          \tikzstyle{input}=[shape=circle,draw=black!50,top color=white,bottom color=black!50]
          \tikzstyle{output}=[shape=circle,draw=black!50,top color=white,bottom color=black!50]
          \tikzstyle{disturbance}=[shape=circle,fill=red!20!blue,draw=none,text=white,inner sep=2pt]
          \node[node] (x1)                 {$x_1$};
          \node[output] (y1) [below of =x1]  {$y_1$};
		  \node[input] (u1) [above of =x1]  {$u_1$};
		  \node[node] (x2) [right of =x1]  {$x_2$};
		  \node[output] (y2) [below right of =x2] {$y_2$};
		  \node[node] (x3) [above right of = y2] {$x_3$};
		  \node[node] (x4) [right of =x3] {$x_4$};
		  \node[input] (u2) [above of =x4]  {$u_2$};
		  \node[output] (y3) [below of =x4]  {$y_3$};

		  \path (u1) edge (x1)
		  		(x1) edge (y1)
				(u2) edge (x4)
				(x4) edge (y3)
				(x2) edge (x1)
				(x3) edge (x4)
				(x3) edge (x2)
				(x2) edge (x3)
				(x2) edge (y2)
				(x3) edge (y2)
		  ;
		  \draw [dashed] (-0.6,2.2) rectangle (0.6,-2.1)
		  				 (0.7,2.2) rectangle (4.4,-2.1)
		  				 (4.5,2.2) rectangle (5.7,-2.1)
		  ;
		  \node at (0.4,2.0) {$\Ncal_1$};
		  \node at (2.6,2.0) {$\Ncal_2$};
		  \node at (4.8,2.0) {$\Ncal_3$};
		  \end{tikzpicture}
		\caption{Example of an aggregation of a BCN with $4$ state nodes, $2$ input nodes, and $3$ output nodes.}
		\label{fig14:BCNnetworkgraph}
      \end{center}
	\end{figure}

\end{example}

\begin{example}\label{exam2:reconstructibility_aggregation}
	Consider the following BCN corresponding to Fig. \ref{fig15:BCNnetworkgraph}.
		\begin{equation}\label{eqn4:reconstructibility_aggregation}
		\begin{split}
			&\Sig_1:\left\{  
			\begin{split}
				&x_1(t+1)=u_1(t)\oplus x_2(t),\\
				&y_1(t)=x_1(t),
			\end{split}
			\right.\\
			&\Sig_2:\left\{  
			\begin{split}
				&x_2(t+1)=x_1(t)\odot x_3(t),\\
				&x_3(t+1)=x_4(t)\odot x_2(t),\\
				&y_2(t)=x_2(t)\odot x_3(t),
			\end{split}
			\right.
			\\
			&\Sig_3:\left\{  
			\begin{split}
				&x_4(t+1)=u_2(t)\oplus x_3(t),\\
				&y_3(t)=x_4(t),
			\end{split}
			\right.
			\\
		\end{split}
	\end{equation}
	where $t=0,1,\dots$; $x_i(t),u_j(t),y_k(t)\in\D$,
	$i\in[1,4]$, $j\in[1,2]$, $k\in[1,3]$.
	
	The cyclic aggregation shown in Fig. \ref{fig15:BCNnetworkgraph} also satisfies Assumption
	\ref{assu1:observability_aggregation}. In Example \ref{exam2:observability_aggregation}, we have 
	shown that $\Sig_1$ and $\Sig_3$ are both observable, hence they are both reconstructible.
	For $\Sig_2$, in its RWPG, there is a self-loop $\{10,01\}\xrightarrow[]{11}\{10,01\}$,
	then by Proposition 
	\ref{prop1:reconstructibility_aggregation}, $\Sig_2$ is not reconstructible.
	Now consider the whole BCN \eqref{eqn4:reconstructibility_aggregation}. Similar to Example 
	\ref{exam1:reconstructibility_aggregation}, we have the whole BCN is observable, and hence reconstructible.

		\begin{figure}
		\begin{center}
        \begin{tikzpicture}[->,>=stealth,node distance=1.5cm]
          \tikzstyle{node}=[shape=circle,draw=black!50,top color=white,bottom color=black!50]
          \tikzstyle{input}=[shape=circle,draw=black!50,top color=white,bottom color=black!50]
          \tikzstyle{output}=[shape=circle,draw=black!50,top color=white,bottom color=black!50]
          \tikzstyle{disturbance}=[shape=circle,fill=red!20!blue,draw=none,text=white,inner sep=2pt]
          \node[node] (x1)                 {$x_1$};
          \node[output] (y1) [below of =x1]  {$y_1$};
		  \node[input] (u1) [above of =x1]  {$u_1$};
		  \node[node] (x2) [right of =x1]  {$x_2$};
		  \node[output] (y2) [below right of =x2] {$y_2$};
		  \node[node] (x3) [above right of = y2] {$x_3$};
		  \node[node] (x4) [right of =x3] {$x_4$};
		  \node[input] (u2) [above of =x4]  {$u_2$};
		  \node[output] (y3) [below of =x4]  {$y_3$};

		  \path (u1) edge (x1)
		  		(x1) edge (y1)
				(u2) edge (x4)
				(x4) edge (y3)
				(x2) edge (x1)
				(x3) edge (x4)
				(x3) edge (x2)
				(x2) edge (x3)
				(x2) edge (y2)
				(x3) edge (y2)
				(x1) edge (x2)
				(x4) edge (x3)
		  ;
		  \draw [dashed] (-0.6,2.2) rectangle (0.6,-2.1)
		  				 (0.7,2.2) rectangle (4.4,-2.1)
		  				 (4.5,2.2) rectangle (5.7,-2.1)
		  ;
		  \node at (0.4,2.0) {$\Ncal_1$};
		  \node at (2.6,2.0) {$\Ncal_2$};
		  \node at (4.8,2.0) {$\Ncal_3$};
		  \end{tikzpicture}
		\caption{Example of an aggregation of a BCN with $4$ state nodes, $2$ input nodes, and $3$ output nodes.}
		\label{fig15:BCNnetworkgraph}
      \end{center}
	\end{figure}

\end{example}

\begin{example}\label{exam3:reconstructibility_aggregation}
	Consider the following BCN corresponding to Fig. \ref{fig16:BCNnetworkgraph}.
	\begin{equation}\label{eqn5:reconstructibility_aggregation}
		\begin{split}
			&\Sig_1:\left\{  
			\begin{split}
				&x_1(t+1)=u_1(t)\vee x_1(t),\\
				&x_2(t+1)=x_1(t)\bar\vee x_4(t),\\
				&y_1(t)=x_2(t),
			\end{split}
			\right.\\
			&\Sig_2:\left\{  
			\begin{split}
				&x_3(t+1)=x_1(t)\bar\vee x_4(t),\\
				&x_4(t+1)=u_2(t)\vee x_4(t),\\
				&y_2(t)=x_3(t),
			\end{split}
			\right.
			\\
		\end{split}
	\end{equation}
	where $t=0,1,\dots$; $x_i(t),u_j(t),y_k(t)\in\D$, $i\in[1,4]$, $j,k\in[1,2]$.
	
	The cyclic aggregation shown in Fig. \ref{fig16:BCNnetworkgraph} satisfies Assumption 
	\ref{assu1:observability_aggregation}.
	One directly sees that both $\Sig_1$ and $\Sig_2$ are observable,
	hence they are also reconstructible.
	In the RWPG of the whole BCN \eqref{eqn5:reconstructibility_aggregation}, there is a self-loop 
	$\{1001,0000\}\xrightarrow[]{00}\{1001,0000\}$, then by Proposition 
	\ref{prop1:reconstructibility_aggregation}, \eqref{eqn5:reconstructibility_aggregation} is not reconstructible.

		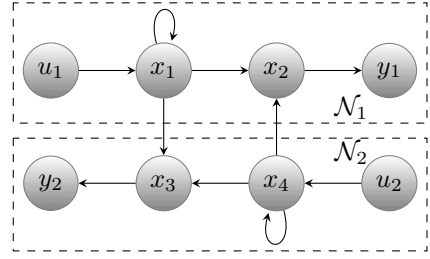
\begin{figure}
		\begin{center}
        \begin{tikzpicture}[->,>=stealth,node distance=1.5cm]
          \tikzstyle{node}=[shape=circle,draw=black!50,top color=white,bottom color=black!50]
          \tikzstyle{input}=[shape=circle,draw=black!50,top color=white,bottom color=black!50]
          \tikzstyle{output}=[shape=circle,draw=black!50,top color=white,bottom color=black!50]
          \tikzstyle{disturbance}=[shape=circle,fill=red!20!blue,draw=none,text=white,inner sep=2pt]
          \node[node] (x1)                 {$x_1$};
          \node[node] (x2) [right of =x1]  {$x_2$};
		  \node[node] (x3) [below of =x1]  {$x_3$};
		  \node[node] (x4) [right of =x3]  {$x_4$};
		  \node[output] (y2) [left of =x3] {$y_2$};
		  \node[input] (u1) [left of =x1] {$u_1$};
		  \node[output] (y1) [right of =x2] {$y_1$};
		  \node[input] (u2) [right of =x4] {$u_2$};
		  \path (x1) edge (x2)
		  		(x1) edge (x3)
				(x4) edge (x2)
		  		(x4) edge (x3)
				(u1) edge (x1)
				(x2) edge (y1)
				(x3) edge (y2) 
				(u2) edge (x4)
				(x1) [loop above] edge (x1)
				(x4) [loop below] edge (x3)
				;
		  \draw [dashed] (-2.0,0.9) rectangle (3.5,-0.7)
		  				 (-2.0,-0.9) rectangle (3.5,-2.4)
		  ;
		  \node at (2.5,-0.5) {$\Ncal_1$};
		  \node at (2.5,-1.1) {$\Ncal_2$};
		  \end{tikzpicture}
		\caption{Example of an aggregation of a BCN with $4$ state nodes, $2$ input nodes, and $2$ output nodes.}
		\label{fig16:BCNnetworkgraph}
      \end{center}
	\end{figure}
\end{example}

Second we show for acyclic aggregations, similar to observability (Theorem \ref{thm1:observability_aggregation}),
all resulting sub-BCNs being reconstructible implies the 
whole BCN also being reconstructible. The similar proof is omitted.

\begin{theorem}\label{thm1:reconstructibility_aggregation}
	Consider a BCN \eqref{BCN1:ObservabilityAggregation} that has
	an acyclic aggregation \eqref{partition:observability_aggregation} satisfying Assumption 
	\ref{assu1:observability_aggregation}.
	If all resulting sub-BCNs 
	are reconstructible then
	the whole BCN is also reconstructible.
\end{theorem}

\section{Application}\label{sec:application}

In this section, we apply  our method to study the observability/reconstructibility
of the BCN T-cell receptor kinetics model
\cite{Klamt2006TCellReceptor}.

T-cells are a type of white blood cells known as lymphocytes. These white blood cells play a central role
in adaptive immunity and enable the immune system to mount specific immune responses.
T-cells have the ability to recognize potentially dangerous agents and subsequently initiate
an immune reaction against them. They do so by using T-cell receptors to detect foreign antigens
bound to major histocompatibility complex molecules, and then activate, through a signaling cascade,
several transcription factors. These transcription factors, in turn, influence the cell's fate such as proliferation.
For the details, we refer the reader to \cite{Klamt2006TCellReceptor}. The BCN T-cell receptor kinetics model
given in \cite{Klamt2006TCellReceptor} is shown in Tab. \ref{tab1:BCNnetworkgraph}, its network graph
is shown in Fig. \ref{fig10:BCNnetworkgraph}. In Fig. \ref{fig10:BCNnetworkgraph}, there exist $3$ input nodes,
and $37$ state nodes. Hence the model has $2^3$ inputs, and $2^{37}$ states.
It is almost impossible to use a PC to deal with such a large BCN directly.
We next use the aggregation method to deal with it. 
In order to obtain the initial state of the BCN, one must choose some state nodes to observe,
since there exists no output node in the network.
In this sense, one should first choose some state nodes, and then assign each of these chosen state nodes
one new output node as its child such that each of these new added output nodes has only one parent.
Then one can obtain the values of these chosen state nodes by observing their corresponding output nodes.
The chosen state nodes and their corresponding output nodes
are represented as the nodes with shadows and their
shadows, respectively, in Fig. \ref{fig10:BCNnetworkgraph}.
Specifically, we choose the set consisting of the following $16$ state nodes:
\begin{equation}
	\begin{split}
		&\{TCRbind, cCbl, PAGCsk, Rlk, TCRphos,SLP76,\\
		&Itk, Grb2Sos, PLCg(bind), CRE, AP1, NFkB,\\
		&NFAT, Fos, Jun, RasGRP1\}.
	\end{split}
	\label{eqn8:observability_aggregation}
\end{equation}
Actually, these chosen state nodes
can make the whole BCN observable, and if any one of them cannot be observed, the whole BCN is not observable.
In what follows, we prove this conclusion, and illustrate the process of choosing these state nodes
by using the previous main results.

We assume that it is not known which nodes in Fig. \ref{fig10:BCNnetworkgraph} are chosen to be observed,
and next illustrate the process of looking for them.
First of all, $CRE$, $AP1$, $NFkB$, and $NFAT$ must be chosen to be observed, because they have no children,
and if any one of them cannot be observed, the whole BCN is not observable. 
By observing these observed nodes, the initial values of 
$CREB$, $Rsk$, $ERK$, $MEK$, $Raf$, $Ras$, $IkB$, $IKKbeta$, $PKCth$,
$DAG$, $Calcin$, $Ca$, $IP3$, and $PLCg(act)$ can be obtained. Taking $NFAT$ for example, by Tab. \ref{tab1:BCNnetworkgraph},
we have $Calcin(0)=NFAT(1)$, 
$Ca(0)=Calcin(1)=NFAT(2)$, $IP3(0)=NFAT(3)$, and $PLCg(act)(0)=NFAT(4)$, then one can obtain $NFAT(0)$,
$Calin(0)$, $Ca(0)$, $IP3(0)$, and $PLCg(act)(0)$
by observing $NFAT(0)$, $NFAT(1)$, $NFAT(2)$, $NFAT(3)$, and $NFAT(4)$, respectively.
Secondly, since $AP1(t+1)=Fos(t)\wedge Jun(t)$
(cf. Tab. \ref{tab1:BCNnetworkgraph}) and both $Fos$ and $Jun$ only affect $AP1$  (cf. Fig. 
\ref{fig10:BCNnetworkgraph}), 
$Fos$ and $Jun$ must be chosen to be observed. This is because if any one of $Fos$ and $Jun$ 
cannot be observed, say, $Fos$, then $Jun(0)=0$ implies that $AP1(1)=0$ no matter what $Fos(0)$ is
and $Fos(0)$ cannot be obtained forever,
i.e., the whole BCN is not observable. 
By observing $Jun$, the initial values of $JNK$ and $SEK$
can be obtained. Thirdly, since $Ras(t+1)=Grb2Sos(t)\vee RasGRP1(t)$ (also cf. Tab. \ref{tab1:BCNnetworkgraph}) 
and $Ras$ is the unique child of both $Grb2Sos$ and $RasGRP1$ (also cf. Fig. \ref{fig10:BCNnetworkgraph}),
similar to the case for $Fos$ and $Jun$, we have
$Grb2Sos$ and $RasGRP1$ must be chosen to be observed, or the whole BCN is not observable.
By observing $Grb2Sos$, the initial values of $LAT$ and $ZAP70$ can be obtained.
Fourthly, we give an acyclic aggregation for the BCN as shown in Fig. \ref{fig10:BCNnetworkgraph}
(also cf. Fig. \ref{fig11:BCNnetworkgraph}), and appoint that for each state node chosen to be observed,
its corresponding new added output node belongs to the same part as the state node.
By the above analysis, the resulting sub-BCNs $\Sig_3,\Sig_4$, and $\Sig_5$ corresponding to
subgraphs $\Ncal_3,\Ncal_4$, and $\Ncal_5$ are observable. Next, we look for the minimal number of
state nodes to be observed in subgraphs
$\Ncal_1$ and $\Ncal_2$ to make the corresponding sub-BCNs observable by using Proposition
\ref{prop1:observability_aggregation}. Since subgraph $\Ncal_1$ has $3$ input nodes, and $8$ state nodes,
and $\Ncal_2$ has fewer input nodes and state nodes, we can use Proposition \ref{prop1:observability_aggregation} to 
deal with them.
For $\Ncal_2$, if we choose all state nodes to be observed, then obviously
the resulting sub-BCN is observable. Since we want to know which node is necessary for the resulting BCN to be observable, we 
choose any $5$ of the $6$ state nodes in $\Ncal_2$ to be observed,
and verify the observability of the resulting sub-BCNs by using Proposition 
\ref{prop1:observability_aggregation}. After verifying the observability of these $6$ sub-BCNs one by one,
we find that $SLP76$, $Itk$, $Grb2Sos$, and $PLCg(bind)$ are necessary for these sub-BCNs to be observable,
and the other two nodes are not.
Also by Proposition \ref{prop1:observability_aggregation}, we obtain that if we choose only 
$SLP76$, $Itk$, $Grb2Sos$, and $PLCg(bind)$ to be observed, then the resulting sub-BCN, denoted by $\Sig_2$,
is observable.
Hence the set of these $4$ nodes is the unique minimal set of nodes making the resulting sub-BCN observable.
Based on this, we choose $SLP76$, $Itk$, $Grb2Sos$, and $PLCg(bind)$ to
be observed in $\Ncal_2$. For $\Ncal_1$, using the same method as for dealing with $\Ncal_2$, by Proposition
\ref{prop1:observability_aggregation} we find that $TCRbind$, $cCbl$, $PAGCsk$, $Rlk$, and $TCRphos$
are necessary for the resulting sub-BCNs to be observable, and the other $3$ nodes are not.
We also obtain that if we choose only $TCRbind$, $cCbl$, $PAGCsk$, $Rlk$, and $TCRphos$ to be observed, 
then the resulting sub-BCN, denoted by $\Sig_1$, is observable.
Hence the set of these $5$ nodes is the unique minimal set
of nodes making the resulting sub-BCN observable.

Until now we find all the $16$ state nodes 
shown in \eqref{eqn8:observability_aggregation}. Next we prove that if we choose only these $16$ nodes
to be observed, then the whole BCN is observable. Note in this sense,
the acyclic aggregation shown in Fig. \ref{fig10:BCNnetworkgraph} satisfies Assumption \ref{assu1:observability_aggregation},
and we have proved that all sub-BCNs $\Sig_i$ are observable, $i\in[1,5]$, then by Theorem 
\ref{thm1:observability_aggregation}, the whole BCN is observable.

To finish this part, we show that if any one of these 16 nodes in \eqref{eqn8:observability_aggregation}
cannot be observed, then the whole BCN is not observable. Previously we have shown that 
$CRE$, $AP1$, $Jun$, $Fos$, $NFkB$, $NFAT$, $RasGRP1$, and $Grb2Sos$ are necessary for the whole BCN
to be observable, so if any one of these nodes is not observed, the whole BCN is not observable.
Now consider $Rlk$, $Itk$ and $PLCg(bind)$. By Tab. \ref{tab1:BCNnetworkgraph} and Fig. \ref{fig10:BCNnetworkgraph},
we have
$PLCg(act)(t+1)=PLCg(bind)(t)\wedge SLP76(t)\wedge ZAP70(t)\wedge (Itk(t)\vee Rlk(t))$,
all of $Rlk$, $Itk$ and $PLCg(bind)$ affect only $PLCg(act)$.
If $Rlk$ (resp. $Itk$, $PLCg(bind)$) is not observed and
$SLP76(0)=0$, then $PLCg(act)(1)=0$ no matter what $Rlk(0)$ (resp. $Itk(0)$,
$PLCg(bind)(0)$) is, i.e., $Rlk(0)$ (resp. $Itk(0)$,
$PLCg(bind)(0)$)  cannot be obtained forever,
and hence the whole BCN is not observable.
Consider $SLP76$. We have $SLP76$ affect only $Itk$ and $PLCg(act)$,
$PLCg(act)(t+1)=PLCg(bind)(t)\wedge SLP76(t)\wedge ZAP70(t)\wedge (Itk(t)\vee Rlk(t))$, and
$Itk(t+1)=SLP76(t)\wedge ZAP70(t)$. If $SLP76$ is not observed, and $ZAP70(0)=0$, then
$PLCg(act)(1)=Itk(1)=0$ no matter what $SLP76(0)$ is, i.e., $SLP76(0)$ cannot be obtained forever,
and the whole BCN is not observable either. For $TCRphos$, we have $TCRphos$ affects only $ZAP70$,
and $ZAP70(t+1)=(\neg cCbl(t))\wedge Lck(t)\wedge TCRphos(t)$.
For $PAGCsk$, we have $PAGCsk$ affects only $Lck$, and $Lck(t+1)=(\neg PAGCsk(t))\wedge CD8(t)\wedge CD45(t)$.
Similarly we have if either $TCRphos$ or $PAGCsk$ cannot be observed, then the whole BCN is not observable.
For $cCbl$, we have $cCbl$ affects only $TCRbind$ and $ZAP70$, 
$TCRbind(t+1)=(\neg cCbl(t))\wedge TCRlig(t)$, and $ZAP70(t+1)=(\neg cCbl(t))\wedge Lck(t)\wedge TCRphos(t)$.
If $cCbl$ is not observed, and $TCRlig(0)=TCRphos(0)=0$, then $TCRbind(1)=ZAP70(1)=0$ no matter what
$cCbl(0)$ is,  i.e., $cCbl(0)$ cannot be obtained forever,
and the whole BCN is not observable either.
Finally consider $TCRbind$. $TCRbind$ affects only $Fyn$, $TCRphos$, and $PAGCsk$. We have
$Fyn(t+1)=CD45(t)\wedge(Lck(t)\vee TCRbind(t))$, $TCRphos(t+1)=Fyn(t)\vee(Lck(t)\wedge TCRbind(t))$,
and $PAGCsk(t+1)=Fyn(t)\vee(\neg TCRbind(t))$. If $TCRbind$ is not observed, $CD45(0)=0$, and $Fyn(0)=1$,
then $Fyn(1)=0$, $TCRphos(1)=PAGCsk(1)=1$ no matter what $TCRbind(0)$ is, i.e., $TCRbind(0)$ cannot be
obtained forever, and then the whole BCN is not observable. This part has been finished.
In addition, note that if $Fyn$ is not observed,
we cannot obtain that the whole BCN is not observable by using similar procedure. 
This is because $Fyn$ affects only $PAGCsk$ and $TCRphos$, $PAGCsk(t+1)=Fyn(t)\vee(\neg TCRbind(t))$,
$TCRphos(t+1)=Fyn(t)\vee(Lck(t)\wedge TCRbind(t))$; if $TCRbind(0)=0$, then no matter what $Lck(0)$ is,
$TCRphos(1)=Fyn(0)$; if $TCRbind(0)=1$, then  no matter what $Lck(0)$ is, $PAGCsk(1)=Fyn(0)$.
That is, in both cases, $Fyn(0)$ can be observed. This procedure is not sufficient to prove that the whole
BCN is observable either, so the aggregation method and Proposition \ref{prop1:observability_aggregation} are necessary.

On the other hand, a weaker type of observability (i.e., \cite[Def. 5]{Zhang2014ObservabilityofBCN},
not equivalent to the one studied in this paper)
of BCNs is charaterized in \cite{Li2015ControlObservaBCN} by using an algebraic method,
and it is proved that for the BCN T-cell receptor kinetics 
model, the unique minimal set of nodes making the whole BCN observable is 
\begin{equation}
	\begin{split}
		&\{TCRbind, cCbl, PAGCsk, Rlk, TCRphos,SLP76,\\
		&Itk, Grb2Sos, PLCg(bind), CRE, AP1, NFkB,\\
		&NFAT, Fos, Jun, RasGRP1\},
	\end{split}
	\label{eqn9:observability_aggregation}
\end{equation}
which is a proper subset of \eqref{eqn8:observability_aggregation}, and does not contain $cCbl$
or $PAGCsk$.

Next we study the reconstructibility of the BCN T-cell receptor kinetics model.
Obviously, if we choose the $16$ state nodes shown in 
\eqref{eqn8:observability_aggregation} to be observed, then the whole BCN 
is reconstructible. However, to make the whole BCN reconstructible, we do not need to observe so many state nodes.
In order to use the aggregation method to characterize reconstructibility, we give a new acyclic aggregation for
its network graph as shown in Fig. \ref{fig17:BCNnetworkgraph}.

For $\Ncal_1$, we have that for each state node $x$ in $\Ncal_1$, if $x$ is not observed
and all other state nodes in $\Ncal_1$ are observed, then the resulting sub-BCN $\Sig_1$ is reconstructible.
That is, none of state nodes of $\Ncal_1$ is necessary for $\Sig_1$ to be reconstructible.
$\Ncal_5$ has the same property as $\Ncal_1$. We also have that for each state node $x$ in $\Ncal_2$ (resp. $\Ncal_3$,
$\Ncal_{4_1}$, $\Ncal_{4_2}$), if $x$ is observed and all other state nodes are not observed, then the
resulting sub-BCN $\Sig_2$ (resp. $\Sig_3$, $\Sig_{4_1}$, $\Sig_{4_2}$) is reconstructible.
Besides, for $\Ncal_5$, if only $PKCth$ is observed, then the resulting sub-BCN $\Sig_5$ is reconstructible,
since $DAG(t)=PKCth(t+1)$, $SEK(t)=PKCth(t-1)$, $IKKbeta(t)=PKCth(t-1)$,
$JNK(t)=SEK(t-1)=PKCth(t-2)$, $Jun(t)=JNK(t-1)=PKCth(t-3)$,
$IkB(t)=\neg IKKbeta(t-1)=\neg PKCth(t-2)$, $NFkB(t)=\neg IkB(t-1)=PKCth(t-3)$, i.e., 
after time step $3$, the states of all nodes in $\Ncal_5$ can be observed by observing $PKCth$.

We next assume that any state node in \eqref{eqn18:observability_aggregation} is observed,
and the other state nodes in Fig. \ref{fig17:BCNnetworkgraph} are not observed.
\begin{equation}
	\begin{split}
		&\{TCRbind, cCbl, PAGCsk, Rlk, TCRphos,\\
		&x_2,x_3,x_{4_1},x_{4_2},PKCth\},
	\end{split}
	\label{eqn18:observability_aggregation}
\end{equation}
where $x_2\in \Ncal_2$, $x_3\in\Ncal_3$, $x_{4_1}\in\Ncal_{4_1}$, and $x_{4_2}\in\Ncal_{4_2}$
are arbitrarily chosen. 

In this case, the acyclic aggregation of Fig. \ref{fig17:BCNnetworkgraph} satisfies Assumption 
\ref{assu1:observability_aggregation}. We have shown that the resulting sub-BCNs $\Sig_1$
is observable, hence reconstructible. We also have that the resulting sub-BCNs
$\Sig_2$, $\Sig_3$, $\Sig_{4_1}$, $\Sig_{4_2}$, and $\Sig_5$ are reconstructible. Then by Theorem 
\ref{thm1:reconstructibility_aggregation}, the whole BCN is reconstructible.
Compared to observing at least $16$ state nodes to make the whole BCN be observable,
in order to make the whole BCN be reconstructible, we only need to 
observe $10$ state nodes.


\begin{table*}
	\centering
	\begin{tabular}{cccccc}
		Nodes & Boolean rule & Nodes & Boolean rule & Nodes & Boolean rule \\
		\hline\hline
		CD8 & Input & Gads & LAT & PKCth & DAG\\\hline
		CD45 & Input & Grb2Sos & LAT & PLCg(act) & 
		\quad\begin{tabular}{r}
			(Itk$\wedge$PLCg(bind)$\wedge$SLP76$\wedge$ZAP70)\\
			$\vee$(PLCg(bind)$\wedge$Rlk$\wedge$SLP76$\wedge$ZAP70)
		\end{tabular}\!\!\!\!\\\hline
		TCRlig & Input & IKKbeta & PKCth & PAGCsk & Fyn$\vee$($\neg$TCRbind)\\\hline
		AP1 & Fos$\wedge$Jun & IP3 & PLCg(act) & PLCg(bind) & LAT\\\hline
		Ca & IP3 & Itk & SLP76$\wedge$ZAP70 & Raf & Ras\\\hline
		Calcin & Ca & IkB & $\neg$IKKbeta & Ras & Grb2Sos$\vee$ RasGRP1\\\hline
		cCbl & ZAP70 & JNK & SEK & RasGRP1 & DAG$\wedge$PKCth\\\hline
		CRE &  CREB & Jun & JNK & Rlk & Lck \\\hline
		CREB & Rsk & LAT & ZAP70 & Rsk & ERK \\\hline
		DAG & PLCg(act) & Lck & ($\neg$PAGCsk$)\wedge$CD8$\wedge$CD45 & SEK & PKCth\\\hline
		ERK & MEK & MEK & Raf & SLP76 & Gads\\\hline
		Fos & ERK & NFAT & Calcin & TCRbind & ($\neg$cCbl)$\wedge$TCRlig\\\hline
		Fyn & (Lck$\wedge$CD45)$\vee$(TCRbind$\wedge$CD45)
		& NFkB & $\neg$IkB & TCRphos & Fyn$\vee$(Lck$\wedge$TCRbind)\\\hline
		& & & & ZAP70 & ($\neg$cCbl)$\wedge$Lck$\wedge$TCRphos\\
	\end{tabular}
	\caption{Updating rules for the nodes of the T-cell receptor kinetics model \cite{Klamt2006TCellReceptor}.}
	\label{tab1:BCNnetworkgraph}
\end{table*}

\begin{figure}
	\centering
	\begin{tikzpicture}[->,>=stealth,node distance=2.5cm,
		point/.style={circle,inner sep=0pt,minimum size=2pt,fill=red},
		skip loop/.style={to path={-- ++(0,#1) -| (\tikztotarget)}},
		state/.style={
		rectangle,minimum size=6mm,rounded corners=3mm,
		very thin,draw=black!50,
		top color=white,bottom color=black!50,font=\ttfamily\bf},
		input/.style={
		rectangle,
		minimum size=6mm,
		very thin,draw=black!50!black!50, 
		top color=white, 
		bottom color=black!50, 
		font=\itshape\bf},
		output/.style={
		rectangle,rounded corners=5.5mm,drop shadow={opacity=0.5},
		minimum size=6mm,
		very thin,draw=black!50, 
		top color=white, 
		bottom color=black!50, 
		font=\itshape\bf}]
\matrix[row sep=6mm,column sep=5mm] {
\node (CD45) [input] {\tiny CD45};  & \node (CD8) [input] {\tiny CD8}; & \node (TCRlig) [input] {\tiny TCRlig}; \\ 
& & \node (TCRbind) [output] {\tiny TCRbind};  &  \node (cCbl) [output] {\tiny cCbl};\\
\node (Fyn) [state] {\tiny Fyn}; & & \node (PAGCsk) [output] {\tiny PAGCsk}; & \node (Rlk) [output] {\tiny Rlk}; \\
\node (TCRphos) [output] {\tiny TCRphos};  & \node (Lck) [state] {\tiny Lck}; & \node (ZAP70) [state] {\tiny ZAP70}; & &\node (IP3) [state] {\tiny IP3};\\
\node (LAT) [state] {\tiny LAT}; & \node (Gads) [state] {\tiny Gads}; & \node (SLP76) [output] {\tiny SLP76}; &\node (Itk) [output] {\tiny Itk}; & \node (Ca) [state] {\tiny Ca};\\
\node (Grb2Sos) [output] {\tiny Grb2Sos}; & \node (PLCgbind) [output] {\tiny PLCg(bind)}; & & \node (PLCgact) [state] {\tiny PLCg(act)}; & \node (Calcin) [state] {\tiny Calcin};\\
& \node (Ras) [state] {\tiny Ras}; & \node (RasGRP1) [output] {\tiny RasGRP1}; & \node (DAG) [state] {\tiny DAG}; & \node (NFAT) [output] {\tiny NFAT};\\ 
\node (CRE) [output] {\tiny CRE}; &  \node (Raf) [state] {\tiny Raf}; & \node (SEK) [state] {\tiny SEK}; & \node (PKCth) [state] {\tiny PKCth}; & \node (IKKbeta) [state] {\tiny IKKbeta};\\
\node (CREB) [state] {\tiny CREB}; & \node (MEK) [state] {\tiny MEK}; & \node (JNK) [state] {\tiny JNK}; & \node (Jun) [output] {\tiny Jun}; & \node (IkB) [state] {\tiny IkB};\\
\node (Rsk) [state] {\tiny Rsk}; & \node (ERK) [state] {\tiny ERK}; & \node (Fos) [output] {\tiny Fos}; & \node (AP1) [output] {\tiny AP1}; & \node (NFkB) [output] {\tiny NFkB};\\
};

\path 
(CD45) edge (Fyn)
(CD45) edge (Lck)
(CD8) edge (Lck)
(TCRlig) edge (TCRbind)
(cCbl) edge (TCRbind)
(TCRbind) edge (Fyn)
(TCRbind) edge (TCRphos)
(TCRbind) edge (PAGCsk)
(Fyn) edge (PAGCsk)
(Fyn) edge (TCRphos)
(PAGCsk) edge (Lck)
(Lck) edge (Rlk)
(Lck) edge (Fyn)
(Lck) edge (TCRphos)
(Lck) edge [bend right] (ZAP70)
(cCbl) edge (ZAP70)
(ZAP70) edge (cCbl)
(TCRphos) edge [bend right] (ZAP70.-130)
(ZAP70) edge (LAT)
(ZAP70) edge (Itk)
(ZAP70) edge (PLCgact)
(Rlk) edge [bend left] (PLCgact)
(LAT) edge (Gads)
(Gads) edge (SLP76)
(SLP76) edge (Itk)
(SLP76) edge (PLCgact)
(Itk) edge (PLCgact)
(LAT) edge (Grb2Sos)
(LAT) edge (PLCgbind)
(PLCgbind) edge (PLCgact)
(PLCgact.45) edge (IP3)
(PLCgact) edge (DAG)
(RasGRP1) edge (Ras)
(DAG) edge (RasGRP1)
(Ras) edge (Raf)
(DAG) edge (PKCth)
(PKCth) edge (RasGRP1)
(Raf) edge (MEK)
(MEK) edge (ERK)
(ERK) edge (Rsk)
(Rsk) edge (CREB)
(CREB) edge (CRE)
(PKCth) edge (SEK)
(SEK) edge (JNK)
(JNK) edge (Jun)
(Jun) edge (AP1)
(Fos) edge (AP1)
(ERK) edge (Fos)
(IP3) edge (Ca)
(Ca) edge (Calcin)
(Calcin) edge (NFAT)
(PKCth) edge (IKKbeta)
(IKKbeta) edge (IkB)
(IkB) edge (NFkB)
(Grb2Sos) edge (Ras)
;

\draw [dashed] (-4.1,6.0) rectangle (2.6,1.2);
\draw node at (1.8,5.4) {$\Ncal_1$};

\draw[dashed]
			   (-4.1,-6.2) -- (-4.1,-1.4) -- (1.0,-1.4) -- (1.0,-2.5) -- (-0.8,-2.5) --
			   (-0.8,-5.0) -- (2.2,-5.0) -- (2.2,-6.2) -- (-4.0,-6.2) -- cycle;

\draw node at (-2.5,-2.5) {$\Ncal_4$};
			   
\draw[dashed]  (4.1,2.4) -- (2.7,2.4) -- (2.7,0.0) -- (0.7,0.0) -- 
			   (0.7,-1.3) -- (2.7,-1.3) -- (2.7,-2.3) -- (4.1,-2.3) -- cycle;

\draw node at (2.6,-0.35) {$\Ncal_3$};

\draw[dashed]
			   (1.2,-1.4) -- (2.5,-1.4) -- (2.5,-2.4) -- (4.1,-2.4) -- (4.1,-6.2) -- 
			   (2.4,-6.2) -- (2.4,-4.8) -- (-0.6,-4.8) -- (-0.6,-2.6) -- (1.2,-2.6) -- cycle;

 \draw node at (2.5,-2.8) {$\Ncal_5$};

 \draw[dashed]
 				(-4.1,1.1) -- (2.6,1.1) -- (2.6,0.1) -- (0.0,0.1) --
				(0.0,-1.3) -- (-4.1,-1.3) --cycle;

\draw node at (-0.5,0.0) {$\Ncal_2$};

\end{tikzpicture}
\caption{Network graph of the T-cell receptor kinetics model (cf. \cite{Klamt2006TCellReceptor}),
	where rectangles denote input nodes, the
	other nodes denote state nodes, particularly the nodes with shadows are chosen to 
	be observed. The aggregation shown in this figure is acyclic.}
	\label{fig10:BCNnetworkgraph}
\end{figure}
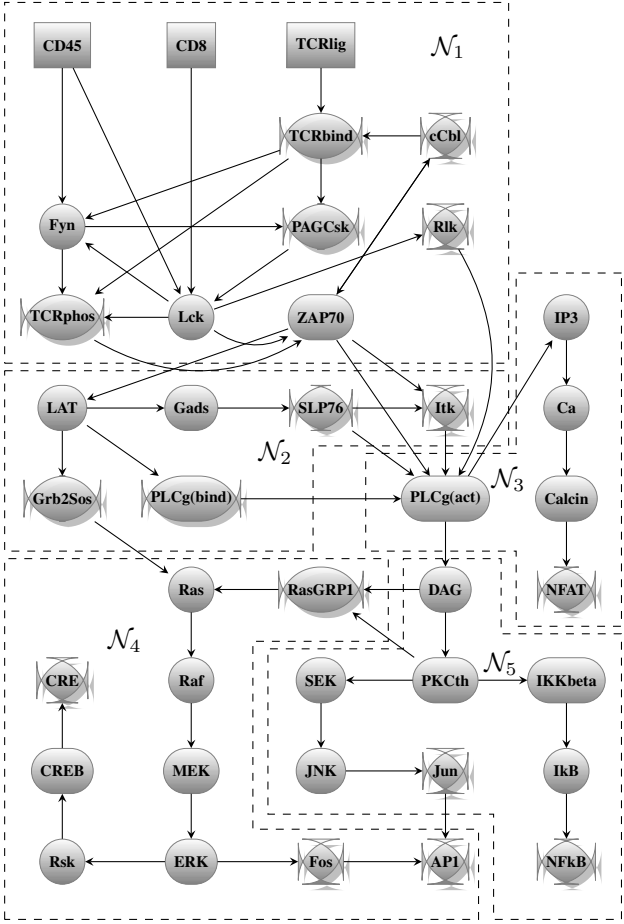

	\begin{figure}
		\begin{center}
        \begin{tikzpicture}[->,>=stealth,node distance=1.5cm]
          \tikzstyle{node}=[shape=circle,draw=black!50,top color=white,bottom color=black!50]
          \tikzstyle{input}=[shape=circle,fill=red!80,draw=none,text=white,inner sep=2pt]
          \tikzstyle{output}=[shape=circle,fill=red!50!blue,draw=none,text=white,inner sep=2pt]
          \tikzstyle{disturbance}=[shape=circle,fill=red!20!blue,draw=none,text=white,inner sep=2pt]
		  \node[node] (n1)                  {$\Ncal_1$};
          \node[node] (n2) [below left of = n1]  {$\Ncal_2$};
          \node[node] (n3) [below right of =n1]  {$\Ncal_3$};
		  \node[node] (n4) [below of =n2]  {$\Ncal_4$};
		  \node[node] (n5) [below of =n3]  {$\Ncal_5$};
		  \path (n1) edge (n2)
		 		(n1) edge (n3)
		  		(n2) edge (n4)
				(n2) edge (n3)
				(n3) edge (n5)
				(n5) edge (n4)
				;
		  \end{tikzpicture}
		  \caption{Aggregation graph corresponding to Fig. \ref{fig10:BCNnetworkgraph}.}
		\label{fig11:BCNnetworkgraph}
      \end{center}
	\end{figure}
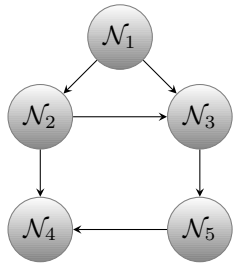

\begin{figure}
	\centering
	\begin{tikzpicture}[->,>=stealth,node distance=2.5cm,
		point/.style={circle,inner sep=0pt,minimum size=2pt,fill=red},
		skip loop/.style={to path={-- ++(0,#1) -| (\tikztotarget)}},
		state/.style={
		rectangle,minimum size=6mm,rounded corners=3mm,
		very thin,draw=black!50,
		top color=white,bottom color=black!50,font=\ttfamily\bf},
		input/.style={
		rectangle,
		minimum size=6mm,
		very thin,draw=black!50!black!50, 
		top color=white, 
		bottom color=black!50, 
		font=\itshape\bf},
		output/.style={
		rectangle,rounded corners=5.5mm,drop shadow={opacity=0.5},
		minimum size=6mm,
		very thin,draw=black!50, 
		top color=white, 
		bottom color=black!50, 
		font=\itshape\bf}]
\matrix[row sep=6mm,column sep=5mm] {
\node (CD45) [input] {\tiny CD45};  & \node (CD8) [input] {\tiny CD8}; & \node (TCRlig) [input] {\tiny TCRlig}; \\ 
& & \node (TCRbind) [state] {\tiny TCRbind};  &  \node (cCbl) [state] {\tiny cCbl};\\
\node (Fyn) [state] {\tiny Fyn}; & & \node (PAGCsk) [state] {\tiny PAGCsk}; & \node (Rlk) [state] {\tiny Rlk}; \\
\node (TCRphos) [state] {\tiny TCRphos};  & \node (Lck) [state] {\tiny Lck}; & \node (ZAP70) [state] {\tiny ZAP70}; & &\node (IP3) [state] {\tiny IP3};\\
\node (LAT) [state] {\tiny LAT}; & \node (Gads) [state] {\tiny Gads}; & \node (SLP76) [state] {\tiny SLP76}; &\node (Itk) [state] {\tiny Itk}; & \node (Ca) [state] {\tiny Ca};\\
\node (Grb2Sos) [state] {\tiny Grb2Sos}; & \node (PLCgbind) [state] {\tiny PLCg(bind)}; & & \node (PLCgact) [state] {\tiny PLCg(act)}; & \node (Calcin) [state] {\tiny Calcin};\\
& \node (Ras) [state] {\tiny Ras}; & \node (RasGRP1) [state] {\tiny RasGRP1}; & \node (DAG) [state] {\tiny DAG}; & \node (NFAT) [state] {\tiny NFAT};\\ 
\node (CRE) [state] {\tiny CRE}; &  \node (Raf) [state] {\tiny Raf}; & \node (SEK) [state] {\tiny SEK}; & \node (PKCth) [state] {\tiny PKCth}; & \node (IKKbeta) [state] {\tiny IKKbeta};\\
\node (CREB) [state] {\tiny CREB}; & \node (MEK) [state] {\tiny MEK}; & \node (JNK) [state] {\tiny JNK}; & \node (Jun) [state] {\tiny Jun}; & \node (IkB) [state] {\tiny IkB};\\
\node (Rsk) [state] {\tiny Rsk}; & \node (ERK) [state] {\tiny ERK}; & \node (Fos) [state] {\tiny Fos}; & \node (AP1) [state] {\tiny AP1}; & \node (NFkB) [state] {\tiny NFkB};\\
};

\path 
(CD45) edge (Fyn)
(CD45) edge (Lck)
(CD8) edge (Lck)
(TCRlig) edge (TCRbind)
(cCbl) edge (TCRbind)
(TCRbind) edge (Fyn)
(TCRbind) edge (TCRphos)
(TCRbind) edge (PAGCsk)
(Fyn) edge (PAGCsk)
(Fyn) edge (TCRphos)
(PAGCsk) edge (Lck)
(Lck) edge (Rlk)
(Lck) edge (Fyn)
(Lck) edge (TCRphos)
(Lck) edge [bend right] (ZAP70)
(cCbl) edge (ZAP70)
(ZAP70) edge (cCbl)
(TCRphos) edge [bend right] (ZAP70.-130)
(ZAP70) edge (LAT)
(ZAP70) edge (Itk)
(ZAP70) edge (PLCgact)
(Rlk) edge [bend left] (PLCgact)
(LAT) edge (Gads)
(Gads) edge (SLP76)
(SLP76) edge (Itk)
(SLP76) edge (PLCgact)
(Itk) edge (PLCgact)
(LAT) edge (Grb2Sos)
(LAT) edge (PLCgbind)
(PLCgbind) edge (PLCgact)
(PLCgact.45) edge (IP3)
(PLCgact) edge (DAG)
(RasGRP1) edge (Ras)
(DAG) edge (RasGRP1)
(Ras) edge (Raf)
(DAG) edge (PKCth)
(PKCth) edge (RasGRP1)
(Raf) edge (MEK)
(MEK) edge (ERK)
(ERK) edge (Rsk)
(Rsk) edge (CREB)
(CREB) edge (CRE)
(PKCth) edge (SEK)
(SEK) edge (JNK)
(JNK) edge (Jun)
(Jun) edge (AP1)
(Fos) edge (AP1)
(ERK) edge (Fos)
(IP3) edge (Ca)
(Ca) edge (Calcin)
(Calcin) edge (NFAT)
(PKCth) edge (IKKbeta)
(IKKbeta) edge (IkB)
(IkB) edge (NFkB)
(Grb2Sos) edge (Ras)
;

\draw [dashed] (-4.1,6.0) rectangle (2.6,1.2);
\draw node at (1.8,5.4) {$\Ncal_1$};

\draw[dashed]
			   (-4.1,-6.2) -- (-4.1,-1.4) -- (-2.8,-1.4) -- (-2.8,-5.0) --
			   (-0.8,-5.0) -- (2.2,-5.0) -- (2.2,-6.2) -- (-4.0,-6.2) -- cycle;

\draw[dashed]
			   (-2.6,-1.4) -- (1.0,-1.4) --  (1.0,-2.5) -- (-0.8,-2.5) -- (-0.8,-4.8) --
			   (-2.6,-4.8) -- cycle;

\draw node at (-2.2,-2.5) {$\Ncal_{4_1}$};

\draw node at (-3.2,-2.5) {$\Ncal_{4_2}$};

\draw[dashed]  (4.1,2.4) -- (2.7,2.4) -- (2.7,0.0) -- (0.7,0.0) -- 
			   (0.7,-1.3) -- (2.7,-1.3) -- (2.7,-2.3) -- (4.1,-2.3) -- cycle;

\draw node at (2.6,-0.35) {$\Ncal_3$};

\draw[dashed]
			   (1.2,-1.4) -- (2.5,-1.4) -- (2.5,-2.4) -- (4.1,-2.4) -- (4.1,-6.2) -- 
			   (2.4,-6.2) -- (2.4,-4.8) -- (-0.6,-4.8) -- (-0.6,-2.6) -- (1.2,-2.6) -- cycle;

 \draw node at (2.5,-2.8) {$\Ncal_5$};

 \draw[dashed]
 				(-4.1,1.1) -- (2.6,1.1) -- (2.6,0.1) -- (0.0,0.1) --
				(0.0,-1.3) -- (-4.1,-1.3) --cycle;

\draw node at (-0.5,0.0) {$\Ncal_2$};

\end{tikzpicture}
\caption{A new acyclic aggregation of the network graph of the T-cell receptor kinetics model
	(cf. \cite{Klamt2006TCellReceptor}),
	where rectangles denote input nodes, the
	other nodes denote state nodes.}
	\label{fig17:BCNnetworkgraph}
\end{figure}
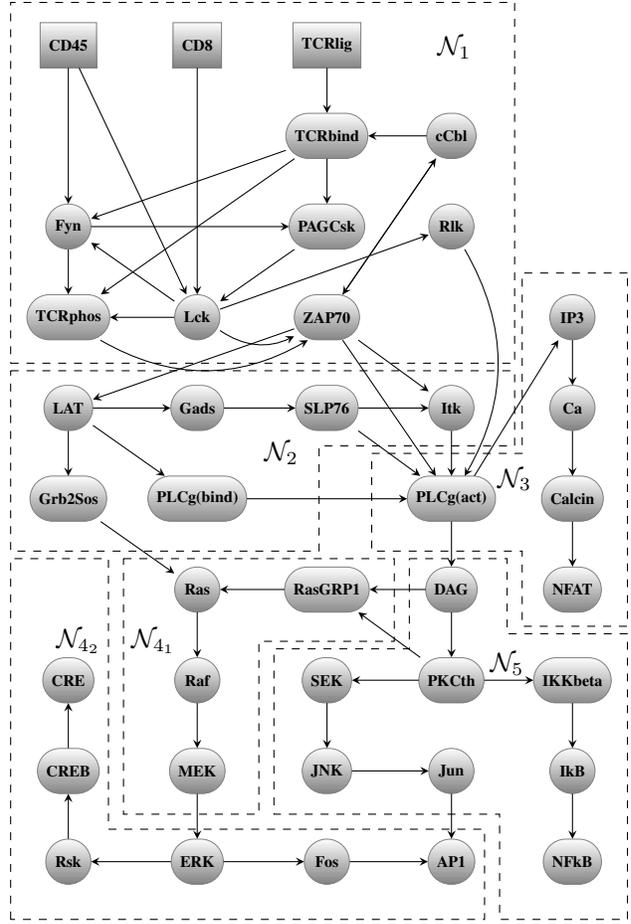

	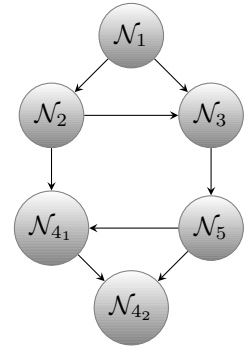
\begin{figure}
		\begin{center}
        \begin{tikzpicture}[->,>=stealth,node distance=1.5cm]
          \tikzstyle{node}=[shape=circle,draw=black!50,top color=white,bottom color=black!50]
          \tikzstyle{input}=[shape=circle,fill=red!80,draw=none,text=white,inner sep=2pt]
          \tikzstyle{output}=[shape=circle,fill=red!50!blue,draw=none,text=white,inner sep=2pt]
          \tikzstyle{disturbance}=[shape=circle,fill=red!20!blue,draw=none,text=white,inner sep=2pt]
		  \node[node] (n1)                  {$\Ncal_1$};
          \node[node] (n2) [below left of = n1]  {$\Ncal_2$};
          \node[node] (n3) [below right of =n1]  {$\Ncal_3$};
		  \node[node] (n41) [below of =n2]  {$\Ncal_{4_1}$};
		  \node[node] (n42) [below right of =n41]  {$\Ncal_{4_2}$};
		  \node[node] (n5) [below of =n3]  {$\Ncal_5$};
		  \path (n1) edge (n2)
		 		(n1) edge (n3)
		  		(n2) edge (n41)
				(n2) edge (n3)
				(n3) edge (n5)
				(n5) edge (n41)
				(n5) edge (n42)
				(n41) edge (n42)
				;
		  \end{tikzpicture}
		  \caption{Aggregation graph corresponding to Fig. \ref{fig17:BCNnetworkgraph}.}
		\label{fig18:BCNnetworkgraph}
      \end{center}
	\end{figure}

\section{Conclusion}\label{sec:conclusion}

In this paper, we used the aggregation method to reduce computational complexity of
verifying the observability/reconstructibility of large-scale BCNs with special structures.
We first defined a special class of aggregations of BCNs that are compatible with observability/reconstructibility,
then we showed that even for this special class of aggregations, the resulting sub-BCNs 
being observable/reconstructible does not imply the whole BCN being observable/reconstructible, and vice versa.
However, for acyclic aggregations in this special class, we proved that 
the resulting sub-BCNs being observable/reconstructible implies the whole BCN being observable/reconstructible,
and acyclic aggregations are equivalent to the cascading aggregations frequently used in the literature.
Finding such aggregations having sufficiently many parts can reduce computational complexity tremendously.
Hence the key point is to find an effective method to look for such aggregations. 

It was proved in \cite{Zhao2013AggregationAlgorithmBNattractor} that the aggregation
consisting of strongly connected components is the finest acyclic aggregation.
However, this aggregation may not be compatible with observability/reconstructibility, since observability/reconstructibility may be 
meaningless for some strongly connected components. Hence how to find acyclic aggregations
that are compatible with observability/reconstructibility is still a challenging problem.
In addition, since the aggregation consisting of strongly connected components is
the finest acyclic aggregation, we can first find it, then by furthermore combining
some strongly connected components to make the new aggregation compatible with observability/reconstructibility.
This is left for further study.


The special aggregation method characterized in this paper can also be used to deal with the observability/reconstructibility of
discrete-time large-scale linear (special classes of nonlinear) control systems over Euclidean spaces
with special network structures.
Taking discrete-time linear time-invariant control systems for example, if an acyclic aggregation satisfying Assumption
\ref{assu1:observability_aggregation} has been found, then one can verify the observability of whole system
via verifying the observability of each part by using the
well known observability rank criterion.
Since given the dimensions of input space, state space, and output space, the set of observable 
linear time-invariant control systems is dense in the set of linear time-invariant control systems,
the aggregation method is feasible.
Further discussion is left for future study.

The main contribution of this paper is showing that one can use the acyclic aggregation method 
to deal with the observability/reconstructibility of large-scale BCNs, which may motive
the study on the observability/reconstructibility of large-scale BCNs based on different types of aggregations.



\end{document}